%% file: convergencecorrelations.tex
\documentclass[a4paper,10pt]{article}

\usepackage[title]{appendix}
\usepackage{amsthm}
\usepackage[mathscr]{eucal}
\usepackage{enumerate}
\usepackage{color}

\usepackage{ut-probability}
\usepackage{math-standard}
\usepackage{ut-graphs}
\usepackage{ulem}

\usepackage{a4wide}

\newcommand\numberthis{\addtocounter{equation}{1}\tag{\theequation}}

\newcommand{\cont}[1]{\widetilde{#1}}

\newcommand{\cdlim}[3]{\left(\left. #1 \right| #2\right) \Rightarrow #3}
\newcommand{\degreesqn}[1]{\mathfrak{D}(#1)}
\newcommand{\bisqn}[1]{\widehat{\mathfrak{D}}(#1)}

\newtheorem{defn}{Definition}[section]
\newtheorem{lem}[defn]{Lemma}
\newtheorem{prop}[defn]{Proposition}
\newtheorem{thm}[defn]{Theorem}

\newtheorem{cond}[defn]{Condition}

\allowdisplaybreaks

\title{Convergence of rank based degree-degree correlations in random directed networks}
\author{Pim van der Hoorn\footnote{University of Twente, w.l.f.vanderhoorn@utwente.nl}, 
	Nelly Litvak\footnote{University of Twente, n.litvak@utwente.nl}}
\date{\today}

\begin{document}

\normalem

\maketitle

\begin{abstract}
	We introduce, and analyze, three measures for degree-degree dependencies, also called degree assortativity,
	in directed random graphs, based on Spearman's rho and Kendall's	tau. We proof statistical 
	consistency of these measures in general random graphs and show that the directed Configuration 
	Model can serve as a null model for	our degree-degree dependency measures. Based on these results
	we argue that the measures we introduce should be preferred over Pearson's correlation coefficients,
	when studying degree-degree dependencies, since the latter has several issues in the case of
	large networks with scale-free degree distributions.
\end{abstract}



\input{introduction}

\input{notations}

\input{rankcorrelations}

\input{randomgraphs}

\input{configurationmodel}

\par \bigskip

\noindent\textbf{Acknowledgments:} \\
We like to thank an anonymous referee for thoroughly reading our manuscript and giving constructive
comments and suggestions for improvement.\\
This work is supported by the EU-FET Open grant NADINE (288956).

\input{appendix}

\bibliographystyle{plain}
\bibliography{bib/convergenceofcorrelations}

\end{document}

%% file: introduction.tex
\section{Introduction}\label{sec:introduction}

This paper investigates statistical consistency of rank correlation measures for dependencies between 
in- and/or out-degrees on both sides of a randomly sampled edge in large directed networks, such as 
the World Wide Web, Wikipedia, or Twitter. These dependencies, also called the assortativity of the 
network, degree correlations, or degree-degree dependencies, represent an important 
topological property of real-world networks, and they have received a vast attention in the 
literature, starting with the work of Newman~\cite{Newman2002,Newman2003}. 

The underlying question that motivates analysis of degree-degree dependencies is whether nodes of 
high in- or out-degree are more likely to be connected to nodes of high or low in- or out-degree. 
These dependencies have been shown to influence many topological features of networks, among others, 
behavior of epidemic spreading~\cite{Boguna2003}, social consensus in Twitter~\cite{liu2014impact}, 
stability of P2P networks under attack~\cite{srivastava2011} and network observability 
\cite{Hasegawa2013}. Therefore, being able to properly measure degree-degree dependencies is 
essential in modern network analysis.  

Given a network, represented by a directed graph, a measurement of degree-degree dependency usually 
consists of computing some expression that is defined by the degrees at both sides of the edges. Here 
the value on each edge can be seen as a realization of some unknown `true' parameter that 
characterizes the degree-degree dependency.

Currently, the most commonly used measure for degree-degree dependencies is a so-called 
\textit{assortativity coefficient}, introduced in \cite{Newman2002,Newman2003}, that computes 
Pearson's correlation coefficient for the degrees at both sides of an edge. However, this
dependency measure suffers from the fact that most real-world networks have highly skewed degree 
distributions, also called \textit{scale-free} distributions, formally described by power laws, 
or more formally, regularly varying distributions. Indeed, when the (in- or out-) degree at the end
of a random edge has infinite variance, then Pearson's coefficient is ill-defined. As a result, 
the dependency measure suggested in \cite{Newman2002,Newman2003} depends on the graph size and 
converges to a non-negative number in the infinite network size limit, as was pointed out in 
several papers~\cite{Dorogovtsev2010,Litvak2013}. The detailed mathematical analysis and examples 
for undirected graphs have been given in~\cite{Litvak2012}, and for directed graphs in our recent 
work~\cite{hoorn2013}. Thus, Pearson's correlation coefficient is not suitable for measuring 
degree-degree dependencies in most real-world directed networks.

The fact that the most commonly used degree correlation measure has obvious mathematical 
flaws, motivates for design and analysis of new estimators. Despite the importance of degree-degree 
dependencies and vast interest from the research community, this remains a largely open problem.

In~\cite{Litvak2012} it was suggested to use a rank correlation measure, Spearman's rho, and it was 
proved that under general regularity conditions, this measure indeed converges to its correct 
population value. Both configuration model and preferential attachment model~\cite{VanDerHofstad2007} 
were proved to satisfy these conditions. In~\cite{hoorn2013} we proposed three rank correlation 
measures, based on Spearman's rho and Kendall's tau, as defined for integer valued random 
varibles, cf.~\cite{Mesfioui2005}, and we compared these measures to Pearson's correlation 
coefficient on Wikipedia graphs for nine different languages. 

In this paper we first prove that, under the convergence assumption of the empirical two-dimensional 
distributions of the degrees on both sides of a random edge, the rank correlations defined in
\cite{hoorn2013} are indeed statistically consistent estimators of degree-degree dependencies. We 
obtain their limiting values in terms of the limiting distributions of the degrees.

Next, we apply our results to the recently developed directed Configuration Model~\cite{chen2013}. 
Roughly speaking, in this model, each node is given a random number of in- and out-bound stubs, that 
are subsequently connected to each other at random. Since multiple edges and self-loops may appear 
as a result of such random wiring, \cite{chen2013} presents two versions of the directed Configuration 
Model. The \textit{repeated} version repeats the wiring until the resulting graph is simple, while 
the \textit{erased} version merges multiple edges and removes self-loops to obtain a simple graph.

We analyze our suggested rank correlation measures in the Repeated and Erased Configuration Model, 
as described in~\cite{chen2013}, and prove that all three measures converge to zero in both models. 
This result is not very surprising for the repeated model, since we connect vertices uniformly at 
random. However, in the erased scenario, the graph is made simple by design, and this might 
contribute to the network showing negative degree-degree dependencies as observed and discussed in, 
for instance,~\cite{Maslov2004,Park2003}. Our result shows that such negative degree-degree 
dependencies vanish for sufficiently large graphs, and thus both flavors of the directed 
Configuration Model can be used as `null model' for our three rank correlation measures. 

By proving consistency of three estimators for degree-degree dependencies in directed networks, and 
providing an easy-to-construct null model for these estimators, this paper makes an important step 
towards assessing statistical significance of degree-degree dependencies in a mathematically 
rigorous way. 

This paper is structured as follows. In Section~\ref{sec:notations} we introduce notations, used 
throughout this paper. Then, in Section~\ref{sec:rankcorrelations}, we prove a general theorem 
concerning statistical consistency of estimators for Spearman's rho and Kendall's tau on 
integer-valued data. This result is applied in Section~\ref{sec:randomgraphs} in the setting of 
random graphs to prove the convergence in the infinite size graph limit of the three degree-degree 
dependency measures from~\cite{hoorn2013}, based on Spearman's rho and Kendall's tau. We analyze 
both the Repeated and Erased Directed Configuration Model in Section~\ref{sec:config}.

%% file: notations.tex
\section{Notations and definitions}\label{sec:notations}

Throughout the paper, if $X$ and $Y$ are random variables we denote their distribution functions by 
$F_X$ and $F_Y$, respectively, and their joint distribution by $H_{X, Y}$.  For integer valued random 
variables $X, Y$ and $k, l \in \Z$ we will often use the following notations:
\begin{align}
	\mathcal{F}_X(k) &= F_X(k) + F_X(k - 1), \label{eq:scriptF}\\
	\mathcal{H}_{X,Y}(k, l) &= H_{X,Y}(k, l) + H_{X,Y}(k - 1, l) + H_{X,Y}(k, l - 1) 
	+ H_{X,Y}(k - 1, l - 1). \label{eq:scriptH}
\end{align}

If $Z$ is a random element, we define the function $F_{X|Z} : \R \times \Omega \to [0, 1]$ by
\[
	F_{X|Z}(x, \omega) = \CExp{\Ind{X \le x}}{Z}(\omega),
\]
where $\Ind{X \le x}$ denotes the indicator of the event $\{\omega : X(\omega) \le x\}$. We 
furthermore define the random variable $F_{X|Z}(Y)$ by 
\[
	F_{X|Z}(Y)(\omega) = F_{X|Z}(Y(\omega), \omega),
\]
and we write $F_{X|Z}(x)$ to indicate the random variable $\CExp{\Ind{X \le x}}{Z}$. With these 
notations it follows that if $X'$ is an independent copy of $X$, then
\begin{align*}
	\CExp{\Ind{X' \le X}}{Z} &= \int_{\R} \int_{\R} \Ind{z \le x} d\Prob{z|Z} d\Prob{x|Z} \\
	&= \int_{\R} \CExp{\Ind{X' \le x}}{Z} d\Prob{x|Z} \\
	&= \CExp{F_{X|Z}(X)}{Z}.
\end{align*}
Using similar definitions for $H_{X,Y|Z}(x, y, \omega)$ and $H_{X,Y|Z}(X, Y)$ we get, if
$(X', Y')$ and $(X'', Y'')$ are independent copies of $(X, Y)$, that
\[
	\CExp{\Ind{X' \le X}\Ind{Y'' \le Y}}{Z} = \CExp{H_{X,Y|Z}(X, Y)}{Z}.
\]
For integer valued random variables $X$ and $Y$, the random variables $\mathcal{F}_{X|Z}(k)$ and 
$\mathcal{H}_{X, Y|Z}(k, l)$ are defined similarly to~\eqref{eq:scriptF} and~\eqref{eq:scriptH}, 
using $F_{X|Z}(k)$ and $H_{X, Y|Z}(k, l)$, respectively.

We introduce the following notion of convergence, related to convergence in distribution.

\begin{defn}\label{def:condconvdist}
	Let $\{X_n\}_{n \in \N}$ and $X$ be random variables and $\{Z_n\}_{n \in \N}$ be a sequence of 
	random elements. We say that $X_n$ converges in distribution to $X$ conditioned on $Z_n$ and 
	write
	\[
		\cdlim{X_n}{Z_n}{X} \quad \text{as } n \to \infty
	\]
	if and only if for all continuous, bounded $h : \R \to \R$
	\[
		\CExp{h(X_n)}{Z_n} \plim \Exp{h(X)} \quad \text{as } n \to \infty.
	\]
\end{defn}

Here $\plim$ denotes convergence in probability. Note that if $h$ is bounded then $\CExp{h(X_n)}
{Z_n}$ is bounded almost everywhere, hence $\lim_{n \to \infty} \Exp{h(X_n)} = \lim_{n \to \infty}
\Exp{\CExp{h(X_n)}{Z_n}} = \Exp{h(X)}$. Therefore, $\cdlim{X_n}{Z_n}{X}$ implies that $X_n 
\Rightarrow X$, where we write $\Rightarrow$ for convergence in distribution. Similar to 
convergence in distribution, it holds that Definition~\ref{def:condconvdist} is equivalent to
\[
	F_{X_n|Z_n}(k) \plim F_X(k) \quad \text{as } n \to \infty, \quad \text{for all } k \in \Z.
\] 

In this paper we use a continuization principle, applied for instance in~\cite{Mesfioui2005}, 
where we transform given discrete random variables in continuous ones. From here on we will work 
with integer valued random variables instead of arbitrary discrete random variables.

\begin{defn}\label{def:cont}
	Let $X$ be an integer valued random variable and $U$ a uniformly distributed random variable on 
	$[0, 1)$ independent of $X$. Then we define the \emph{continuization} of $X$ as
	\[
		\cont{X} = X + U.
	\]  
\end{defn}

We will refer to $U$ as the \emph{continuous part} of $\cont{X}$. We remark that although we have 
chosen $U$ to be uniform we could instead take any continuous random variable on $[0, 1)$ with 
strictly increasing cdf, cf.~\cite{Denuit2005}.

%% file: rankcorrelations.tex
\section{Rank correlations for integer valued random variables}\label{sec:rankcorrelations}

We will use the rank correlations Spearman's rho and Kendall's tau for integer valued random
variables as defined in~\cite{Mesfioui2005}. Below we will state these and rewrite them in terms of
the functions $\mathcal{F}$ and $\mathcal{H}$, defined in~\eqref{eq:scriptF} and~\eqref{eq:scriptH}
respectively. We will then proceed, defining estimators for these correlations and prove that, under 
natural conditions, these converge to the correct value.

\input{rankcorrelations_spearmanrho}

\input{rankcorrelations_kendalltau}

\input{rankcorrelations_convergence}

%% file: rankcorrelations_spearmanrho.tex
\subsection{Spearman's rho}

Given two integer valued random variables $X$ and $Y$, Spearman's rho $\spearman(X, Y)$ is defined as, 
c.f.~\cite{Mesfioui2005}
\begin{align*}
	\spearman(X, Y) &= 3\left(\Prob{X < X', Y < Y''} + \Prob{X \le X', Y < Y''} \right. \\
	&\hspace{20pt}+\left. \Prob{X < X', Y \le Y''} + \Prob{X \le X', Y \le Y''} - 1\right), 
\end{align*}
where $(X', Y')$ and $(X'', Y'')$ are independent copies of $(X, Y)$. We will rewrite this expression, 
starting with a single term:
\begin{align*}
	\Prob{X < X', Y < Y''} &= \Exp{\Ind{X < X'}\Ind{Y < Y''}} \\
	&= 1 - \Exp{\Ind{X' \le X}} - \Exp{\Ind{Y'' \le Y}} + \Exp{\Ind{X' \le X}\Ind{Y'' \le Y}} \\
	&= 1 - \Exp{F_X(X)} - \Exp{F_Y(Y)} + \Exp{F_X(X)F_Y(Y)}.
\end{align*}
If we do the same for the other three terms and use~\eqref{eq:EF} we obtain,
\begin{equation}
	\spearman(X, Y) = 3\Exp{\mathcal{F}_X(X)\mathcal{F}_Y(Y)} - 3.
\label{eq:spearmanxy}
\end{equation}

Since, given two continuous random variables $\mathcal{X}$ and $\mathcal{Y}$, Spearman's rho is 
defined as
\[
	\spearman(\mathcal{X}, \mathcal{Y}) = 12\Exp{F_{\mathcal{X}}(\mathcal{X})F_{\mathcal{Y}}(\mathcal{Y})} 
	- 3,
\]
Lemma~\ref{lem:contjointmoments} now implies that
\begin{align}
	\spearman(X, Y) = \spearman(\cont{X}, \cont{Y}).
	\label{eq:spearmanintuniform}
\end{align}

%% file: rankcorrelations_kendalltau.tex
\subsection{Kendall's tau}

For two continuous random variables $\mathcal{X}$ and $\mathcal{Y}$, Kendall's tau $\tau(\mathcal{X}, 
\mathcal{Y})$ is defined as
\[
	\tau(\mathcal{X}, \mathcal{Y}) = 4\Exp{H_{\mathcal{X}, \mathcal{Y}}(\mathcal{X}, \mathcal{Y})} - 1.
\]
Given two discrete random variables $X$ and $Y$, Kendall's Tau can be written as, c.f.
\cite{Mesfioui2005} Proposition 2.2,
\begin{equation}
	\kendall(X, Y) = \Exp{\mathcal{H}_{X, Y}(X, Y)} - 1.
	\label{eq:kendallxy}
\end{equation}
Similar to Spearman's rho we obtain, using Lemma~\ref{lem:contjointmoments}, that
\begin{align}
	\kendall(X, Y) = \kendall(\cont{X}, \cont{Y}).
	\label{eq:kendallintuniform}
\end{align}
Hence applying the continuization principle from Definition~\ref{def:cont} on $X$ and $Y$ preserves 
both rank correlations. We remark that~\eqref{eq:spearmanintuniform} and~\eqref{eq:kendallintuniform} 
were obtained for arbitrary discrete random variables, using a different approach, in~\cite{Mesfioui2005}.

%% file: rankcorrelations_convergence.tex
\subsection{Convergence for Spearman's rho and Kendall's tau}

Let $\{X_n\}_{n \in \N}$ and $\{Y_n\}_{n \in \N}$ be sequences of integer valued random variables. 
If $(X_n, Y_n) \Rightarrow (X, Y)$, for some integer valued random variables $X$ and $Y$, then 
$\lim_{n \to \infty}\Exp{\mathcal{F}_{X_n}(X_n)\mathcal{F}_{Y_n}(Y_n)} = \Exp{\mathcal{F}_X(X)
\mathcal{F}_Y(Y)}$ which implies that $\lim_{n \to \infty} \spearman(X_n, Y_n) = \spearman(X, Y)$. 
The next theorem generalizes this to the setting of the convergence of $(X_n, Y_n|Z_n)$, of 
Definition~\ref{def:condconvdist}.

\begin{thm}\label{thm:convrankcorr}
	Let $\{X_n\}_{n \in \N}$, $\{Y_n\}_{n \in \N}$ be sequences of integer valued random variables 
	for which there exist	a sequence $\{Z_n\}_{n \in \N}$ of random elements and two integer 
	valued random variables $X$ and $Y$ such that 
	\[
		\cdlim{X_n, Y_n}{Z_n}{(X, Y)} \quad \text{as } n \to \infty.
	\]
	Then, as $n \to \infty$,
	\begin{enumerate}[\upshape i)]
		\item 
		$3\CExp{\mathcal{F}_{X_n|Z_n}(X_n)\mathcal{F}_{Y_n|Z_n}(Y_n)}{Z_n} - 3 \, \plim \spearman(X, Y)$ 
		and
		\item
		$\CExp{\mathcal{H}_{X_n, Y_n|Z_n}(X_n, Y_n)}{Z_n} - 1 \, \plim \kendall(X, Y)$.
	\end{enumerate}
	Moreover, we also have convergence of the expectations:
	\begin{enumerate}[\upshape i)]
		\setcounter{enumi}{2}
		\item 
		$\lim_{n \to \infty} 3\Exp{\mathcal{F}_{X_n|Z_n}(X_n)\mathcal{F}_{Y_n|Z_n}(Y_n)} - 3 = 
		\spearman(X, Y)$ and
		\item
		$\lim_{n \to \infty} \Exp{\mathcal{H}_{X_n, Y_n|Z_n}(X_n, Y_n)} - 1 = \kendall(X, Y)$.
	\end{enumerate}
\end{thm}

\begin{proof}
	Observe first that since $\cdlim{X_n, Y_n}{Z_n}{(X, Y)}$, it follows that for all $k, l \in \Z$, 
	as $n \to \infty$,
	\begin{align}
		F_{X_n|Z_n}(k) &\plim F_{X}(k) \label{eq:thm31_convergence_x}\\
		F_{Y_n|Z_n}(l) &\plim F_{Y}(l) \label{eq:thm31_convergence_y}\\
		H_{X_n, Y_n|Z_n}(k, l) &\plim H_{X, Y}(k, l). \label{eq:thm31_convergence_xy}
	\end{align}
	Moreover, these convergence hold uniformly, since $X$ and $Y$ are integer valued.
	
 	i) Using first~\eqref{eq:spearmanxy} and then applying Lemma~\ref{lem:contjointmoments} 
	and Proposition~\ref{prop:condcontjointprop} we obtain,
	\begin{align}
		&\left|3\CExp{\mathcal{F}_{X_n|Z_n}(X_n)\mathcal{F}_{Y_n|Z_n}(Y_n)}{Z_n} 
			- 3 - \spearman(X, Y)\right| \notag\\
		&=3\left|\CExp{\mathcal{F}_{X_n|Z_n}(X_n)\mathcal{F}_{Y_n|Z_n}(Y_n)}{Z_n} 
			- \Exp{\mathcal{F}_X(X)\mathcal{F}_Y(Y)}\right| \notag\\
		&= 12\left|\CExp{F_{\cont{X}_n|Z_n}(\cont{X}_n)F_{\cont{Y}_n|Z_n}(\cont{Y}_n)}{Z_n} 
			- \Exp{F_{\cont{X}}(\cont{X})F_{\cont{Y}}(\cont{Y})}\right| \notag\\
		&\le 12\left|\CExp{F_{\cont{X}_n|Z_n}(\cont{X}_n)F_{\cont{Y}_n|Z_n}(\cont{Y}_n)}{Z_n}
			- \CExp{F_{\cont{X}}(\cont{X}_n)F_{\cont{Y}}(\cont{Y}_n)}{Z_n}\right| \notag\\
		&\hspace{10pt}+ 12\left|\CExp{F_{\cont{X}}(\cont{X}_n)F_{\cont{Y}}(\cont{Y}_n)}{Z_n} 
			- \Exp{F_{\cont{X}}(\cont{X})F_{\cont{Y}}(\cont{Y})}\right| \notag\\
		&\le 12\sup_{x, y \in \R} \left|F_{\cont{X}_n|Z_n}(x)F_{\cont{Y}_n|Z_n}(y)
			- F_{\cont{X}}(x)F_{\cont{Y}}(y)\right| \label{eq:convrankcorr1}\\
		&\hspace{10pt}+ 12\left|\CExp{F_{\cont{X}}(\cont{X}_n)F_{\cont{Y}}(\cont{Y}_n)}{Z_n} 
			- \Exp{F_{\cont{X}}(\cont{X})F_{\cont{Y}}(\cont{Y})}\right| \label{eq:convrankcorr2}.
	\end{align}
	Because the function $h(x, y) = F_{\cont{X}}(x)F_{\cont{Y}}(y)$ is continuous and bounded,
	\eqref{eq:convrankcorr2} converges in probability to 0. For~\eqref{eq:convrankcorr1} we observe 
	that 
	\begin{align*}
		\left|F_{\cont{X}_n|Z_n}(x)F_{\cont{Y}_n|Z_n}(y) - F_{\cont{X}}(x)F_{\cont{Y}}(y)\right| 
		&\le \left|F_{\cont{X}_n|Z_n}(x)F_{\cont{Y}_n|Z_n}(y) - F_{\cont{X}_n|Z_n}(x)F_{\cont{Y}}(y)\right| \\
		&\hspace{10pt}+ \left|F_{\cont{X}_n|Z_n}(x)F_{\cont{Y}}(y) - F_{\cont{X}}(x)F_{\cont{Y}}(y)\right| \\
		&\le \left|F_{\cont{Y}_n|Z_n}(y) - F_{\cont{Y}}(y)\right| + \left|F_{\cont{X}_n|Z_n}(x) 
		- F_{\cont{X}}(x)\right|.
	\end{align*}
	It now follows that~\eqref{eq:convrankcorr1} converges in probability to 0, since the convergence
	\eqref{eq:thm31_convergence_x} and~\eqref{eq:thm31_convergence_y} are	uniform. 

	ii) Here we again use Lemma~\ref{lem:contjointmoments} and Proposition~\ref{prop:condcontjointprop},
	now combined with~\eqref{eq:kendallxy} to obtain,
	\begin{align*}
		&\left|\CExp{\mathcal{H}_{X_n, Y_n|Z_n}(X_n, Y_n)}{Z_n} - 1 - \kendall(X, Y)\right| \\
		&= \left|\CExp{\mathcal{H}_{X_n, Y_n|Z_n}(X_n, Y_n)}{Z_n} - \Exp{\mathcal{H}_{X, Y}(X, Y)}\right| \\
		&= 4\left|\CExp{H_{\cont{X}_n, \cont{Y}_n|Z_n}(\cont{X}_n, \cont{Y}_n)}{Z_n}
			- \Exp{H_{\cont{X}, \cont{Y}}(\cont{X}, \cont{Y})}\right| \\
		&\le 4\left|\CExp{H_{\cont{X}_n, \cont{Y}_n|Z_n}(\cont{X}_n, \cont{Y}_n)}{Z_n}
			- \CExp{H_{\cont{X}, \cont{Y}}(\cont{X}_n, \cont{Y}_n)}{Z_n}\right| \\
		&\hspace{10pt}+ 4\left|\CExp{H_{\cont{X}, \cont{Y}}(\cont{X}_n, \cont{Y}_n)}{Z_n}
			- \Exp{H_{\cont{X}, \cont{Y}}(\cont{X}, \cont{Y})}\right| \\
		&\le 4 \sup_{x, y \in \R} \left|H_{\cont{X}_n, \cont{Y}_n|Z_n}(x, y)
			- H_{\cont{X}, \cont{Y}}(x, y)\right|
		+ 4\left|\CExp{H_{\cont{X}, \cont{Y}}(\cont{X}_n, \cont{Y}_n)}{Z_n}
			- \Exp{H_{\cont{X}, \cont{Y}}(\cont{X}, \cont{Y})}\right|
	\end{align*}
	The former term converges in probability to 0 because~\eqref{eq:thm31_convergence_xy} holds uniformly, 
	and for the latter this holds since $h(x, y) = H_{\cont{X}, \cont{Y}}(x, y)$ is continuous and bounded.
			
	Since both $\CExp{\mathcal{F}_{X_n|Z_n}(X_n)\mathcal{F}_{Y_n|Z_n}(Y_n)}{Z_n}$ and 
	$\CExp{\mathcal{H}_{X_n, Y_n|Z_n}(X_n, Y_n)}{Z_n}$ are bounded a.e. we obtain iii) and iv) directely 
	from i) and ii), respectively.
\end{proof}

%% file: randomgraphs.tex
\section{Rank correlations for random graphs}\label{sec:randomgraphs}

We now turn to the setting of rank correlations for degree-degree dependencies in random directed 
graphs. We will first introduce some terminology concerning random graphs. Then we will
recall the rank correlations given in~\cite{hoorn2013} and prove statistical consistency of
these measures.

\subsection{Random graphs}

Given a directed graph $G = (V, E)$, we denote by $\left(D^+(v), D^-(v)\right)_{v \in V}$ the degree 
sequence where $D^+$ denotes the out-degree and $D^-$ the in-degree. We adopt the convention, 
introduced in~\cite{hoorn2013}, to index the degree type by $\alpha, \beta \in \{+, -\}$. 
Furthermore, we will use the projections $\pi_\ast, \pi^\ast : V^2 \to V$ to distinguish the source 
and target of a possible edge. That is, if $(v, w) \in V^2$ then $\pi_\ast(v, w) = v$ and $\pi^\ast
(v, w) = w$. When both projections are applicable we will use $\pi$. For $v, w \in V$ we denote by 
$E(v, w) = \{e \in E | \pi_\ast e = v, \pi^\ast e = w\}$ the set of all edges from $v$ to $w$. For 
$e \in V^2$, we write $E(e) = E(\pi_\ast e, \pi^\ast e)$.

Given a set $V$ of vertices we call a graph $G = (V, E)$ random, if for each $e \in V^2$, $|E(e)|$ 
is a random variable. Since $\Ind{e \in E} = \Ind{|E(e)| > 0}$, it follows that the former is also 
a random variable, cf.~\cite{Chung2003} for a similar definition of random graphs using edge 
indicators. Therefore, when we refer to $G$ as a random element it is understood that we refer to 
the random variables $|E(e)|$, for $e \in V^2$.

When $G$ is a random graph, the number of edges in the graph and the degrees of the nodes are 
random variables defined by $\Ind{e \in E}$ and $|E(e)|$, $e\in V^2$: 
\begin{align*}
|E|& = \sum_{e \in V^2} \Ind{e \in E}|E(e)|,\\
D^-(v)&=\sum_{w \in V}\Ind{(w,v) \in E}|E(w,v)|,\quad v\in V,\\
D^+(v)&=\sum_{w \in V}\Ind{(v,w) \in E}|E(v,w)|,\quad v\in V.
\end{align*}

Given a random graph $G = 
(V, E)$ we define a uniformly sampled edge $\uedge_G$ as a two-dimensional random variable on $V^2$ 
such that
\[
	\Prob{\uedge_G = e | G} = \frac{|E(e)|}{|E|}.
\]  
When it is clear which graph we are considering, we will use $\uedge$ instead of $\uedge_G$.
Let $\alpha, \beta \in \{+, -\}$, $k, l \in \N$ and $\pi$ be any of the projections $\pi_\ast$ and 
$\pi^\ast$. Then we define
\begin{align}
	F^\alpha_G(k) &= F_{D^\alpha(\pi (\uedge_G)) | G}(k), \label{eq:Fgraph} \\
	H^{\alpha, \beta}_{G}(k, l) &= H_{D^\alpha( \pi_\ast( \uedge_G)), D^\beta(\pi^\ast( \uedge_G))
	| G}(k, l).	\label{eq:Hgraph}
\end{align}
These functions are the empirical distribution of $D^\alpha (\pi (\uedge_G))$ and the joint 
empirical distribution of $D^\alpha (\pi_\ast (\uedge_G))$ and $D^\beta (\pi^\ast (\uedge_G))$, 
respectively, given the random graph $G$. The functions $\mathcal{F}^\alpha_G$ and $\mathcal{H}^{
\alpha, \beta}_G$ are defined in a similar way as~\eqref{eq:scriptF} and~\eqref{eq:scriptH}, 
using~\eqref{eq:Fgraph} and \eqref{eq:Hgraph}, respectively. In order to keep notations clear, we 
will, when considering both projections $\pi_\ast$ and $\pi^\ast$, always use $\alpha$ to index the 
degree type of the sources and $\beta$ to index the degree type of targets. Moreover, we will often 
write $D^\alpha \pi \uedge_G$ instead of $D^\alpha\left(\pi(\uedge_G)\right)$.

Now we will introduce Spearman's rho and Kendall's tau on random directed graphs and write them in 
terms of the functions~\eqref{eq:Fgraph} and~\eqref{eq:Hgraph}. This way we will be in a setting 
similar to the one of Theorem~\ref{thm:convrankcorr} so that we can utilize this theorem to prove 
statistical consistency of these rank correlations.

\input{randomgraphs_spearmansrho}

\input{randomgraphs_kendallstau}

\input{randomgraphs_convergence}

%% file: randomgraphs_spearmansrho.tex
\subsection{Spearman's Rho}

Spearman's rho measure for degree-degree dependencies in directed graphs, introduced in 
\cite{hoorn2013}, is in fact Pearson's correlation coefficient computed on the ranks of the degrees 
rather than their actual values. In our setting, this definition is ambiguous because the data has 
many ties. For example, if the in-degree of node $v$ is $d$ then we will observe $D^-\pi^*e=d$ for at 
least $d$ edges $e \in E$, plus there will be many more nodes with the same degree. In
\cite{hoorn2013} we consider two possible ways of resolving ties: by assigning a unique rank to each 
tied value uniformly at random, and by assigning the same, average, rank to all tied values. We 
denote the ranks resulting from the random and the average resolution of ties by $R$ and $\bar R$, 
respectively. Formally, for $\alpha, \beta \in \{+, -\}$, we write:
\begin{align}
	R^\alpha \pi_\ast e &= \sum_{f \in E} \Ind{D^\alpha \pi_\ast f + U_f \ge D^\alpha \pi_\ast e + U_e},
	\label{eq:uniformrankalpha}\\
	R^\beta \pi^\ast e &= \sum_{f \in E} \Ind{D^\beta \pi^\ast f + W_f \ge D^\beta \pi^\ast e + W_e},
	\label{eq:uniformrankbeta}
\end{align}  
where $U$, $W$ are independent $|V|^2$ vectors of independent uniform random variables on $[0, 1)$, 
and
\begin{align}
	\overline{R}^\alpha \pi e = \frac{1}{2} + \sum_{f \in E} \Ind{D^\alpha \pi f > D^\alpha \pi e} + \frac{1}{2}
	\Ind{D^\alpha \pi f = D^\alpha \pi e}. \label{eq:avrank}
\end{align}
Then the corresponding two versions of Spearman's rho are defined as follows, cf. \cite{hoorn2013}:
\[
	\spearman_\alpha^\beta(G) = \frac{12 \sum_{e \in E} R^\alpha \pi_\ast(e) R^\beta \pi^\ast(e)
	- 3 |E|(|E| + 1)^2}{|E|^3 - |E|} \quad \text{and}
\]
\[
	\spearmanaverage_\alpha^\beta(G) = \frac{4 \sum_{e \in E} \overline{R}^\alpha \pi_\ast(e) 
	\overline{R}^\beta \pi^\ast(e)
	- |E|(|E| + 1)^2}
	{\text{Var}_\ast(\overline{R}^\alpha)\text{Var}^\ast(\overline{R}^\beta)},
\]
where
\begin{align*}
	\text{Var}_\ast(\overline{R}^\alpha) &= \sqrt{4 \sum_{e \in E} \overline{R}^\alpha \pi_\ast(e)^2 
	- |E|(|E| + 1)^2} \quad \text{and} \\
	\text{Var}^\ast(\overline{R}^\beta) &= \sqrt{4 \sum_{e \in E} \overline{R}^\beta \pi^\ast(e)^2 
	- |E|(|E| + 1)^2}.
\end{align*}

The next proposition relates the random variables $\spearman_\alpha^\beta(G)$ and 
$\spearmanaverage_\alpha^\beta(G)$ to the random variable 
\begin{align}
	\CExp{\mathcal{F}_G^\alpha\left(D^\alpha \pi_\ast \uedge\right)
	\mathcal{F}_G^\beta\left(D^\beta \pi^\ast \uedge\right)}{G}.
	\label{eq:spearmanestimator}
\end{align}

\begin{prop}\label{prop:avuniformranks}
	Let $G = (V, E)$ be a random graph, $\uedge$ an edge on $G$ sampled uniformly at random and 
	$\alpha, \beta \in \{+, -\}$. Then
	\begin{enumerate}[\upshape i)]
		\item 
		$\displaystyle
		\begin{aligned}[t]
			\frac{1}{|E|}\sum_{e \in E} \frac{\overline{R}^\alpha \pi_\ast e}{|E|}\frac{\overline{R}^\beta \pi^\ast e}{|E|} 
			&= \frac{1}{4}\CExp{\mathcal{F}^\alpha_G\left(D^\alpha \pi_\ast \uedge\right)
				\mathcal{F}^\beta_G\left(D^\beta \pi^\ast \uedge\right)}{G} + o_\pr(|E|^{-1}) \quad \text{and}
		\end{aligned}	
		$
		\item 
		$\displaystyle
		\begin{aligned}[t]
			\frac{1}{|E|}\sum_{e \in E} \frac{R^\alpha \pi_\ast e}{|E|}\frac{R^\beta \pi^\ast e}{|E|} 
			&= \frac{1}{4}\CExp{\mathcal{F}^\alpha_G\left(D^\alpha \pi_\ast \uedge\right)
				\mathcal{F}^\beta_G\left(D^\beta \pi^\ast \uedge\right)}{G} + o_\pr(|E|^{-1}).
		\end{aligned}	
		$
	\end{enumerate}
\end{prop}

\begin{proof}
	i) Let $\uedge'$ be an independent copy of $\uedge$ and $e \in V^2$. Then it follows from
	\eqref{eq:avrank} that
	\begin{align*}
		&\frac{\overline{R}^\alpha \pi e}{|E|} = \frac{1}{2|E|} + \sum_{f \in E} \frac{1}{|E|}\Ind{D^\alpha \pi f > 
			D^\alpha \pi e} + \frac{1}{2|E|}\Ind{D^\alpha \pi f = D^\alpha \pi e} \\
		&= 1 + \frac{1}{2|E|} - \frac{1}{2|E|} \sum_{f \in E} \Ind{D^\alpha \pi f \le D^\alpha \pi e} 
			+ \Ind{D^\alpha \pi f \le D^\alpha \pi e - 1} \\
		&= 1 + \frac{1}{2|E|} - \frac{1}{2} \sum_{f \in V^2} \left(\Ind{D^\alpha \pi f \le D^\alpha \pi e} 
			+ \Ind{D^\alpha \pi f \le D^\alpha \pi e - 1}\right) \frac{|E(f)|}{|E|}\\
		&= 1 + \frac{1}{2|E|} - \frac{1}{2}\sum_{f \in V^2} \left(\Ind{D^\alpha \pi f \le D^\alpha \pi e} 
			+ \Ind{D^\alpha \pi f \le D^\alpha \pi e - 1}\right) \Prob{\uedge' = f | G} \\
		&= 1 + \frac{1}{2|E|} - \frac{1}{2}\left(F_G^\alpha\left(D^\alpha \pi e\right) 
			+ F_G^\alpha\left(D^\alpha \pi e - 1\right)\right) \\
		&= 1 + \frac{1}{2|E|} - \frac{1}{2}\mathcal{F}^\alpha_G\left(D^\alpha \pi e\right)
		\numberthis \label{eq:avrankgraph}.
	\end{align*}
	Using a similar expression for $\left(\overline{R}^\beta \pi_\ast e\right)/|E|$ we obtain,
		\begin{align*}
		\frac{1}{|E|}\sum_{e \in E} \frac{\overline{R}^\alpha \pi_\ast e}{|E|}
		\frac{\overline{R}^\beta \pi^\ast e}{|E|} 
		&= \frac{1}{|E|}\sum_{e \in E} \left( 1 + \frac{1}{2|E|} - \frac{1}{2}\mathcal{F}^\alpha_G
		\left(D^\alpha \pi_\ast e\right)\right)\left( 1 + \frac{1}{2|E|} - \frac{1}{2}\mathcal{F}^\beta_G
		\left(D^\beta \pi^\ast e\right)\right)\\
		&= \CExp{\left( 1 + \frac{1}{2|E|} - \frac{1}{2}\mathcal{F}^\alpha_G
		\left(D^\alpha \pi_\ast \uedge\right)\right)\left( 1 + \frac{1}{2|E|} - \frac{1}{2}\mathcal{F}^\beta_G
		\left(D^\beta \pi^\ast \uedge \right)\right)}{G}. 
	\end{align*}
	Rearranging the terms yields
	\begin{align}
		\frac{1}{|E|}\sum_{e \in E} \frac{\overline{R}^\alpha \pi_\ast e}{|E|}
		\frac{\overline{R}^\beta \pi^\ast e}{|E|} &= 
		\frac{1}{4}\CExp{\mathcal{F}^\alpha_G\left(D^\alpha \pi_\ast \uedge\right)
			\mathcal{F}^\beta_G\left(D^\beta \pi^\ast \uedge\right)}{G} \notag \\
		&\hspace{10pt}+ 1 - \frac{1}{2}\CExp{\mathcal{F}^\alpha_G\left(D^\alpha \pi_\ast \uedge\right) 
			+ \mathcal{F}^\beta_G\left(D^\beta \pi^\ast \uedge\right)}{G} + o_\pr(|E|^{-1}).
			\label{eq:avuniformranks1}
	\end{align}
	Since the sum over all average ranks equals $|E|(|E| + 1)/2$, it follows that
	\[
		\frac{1}{2} + \frac{1}{2|E|} = \frac{1}{|E|} \sum_{e \in E} \frac{\overline{R}^\alpha \pi e}{|E|}
		= 1 + \frac{1}{2|E|} - \frac{1}{2}\CExp{\mathcal{F}^\alpha_G\left(D^\alpha \pi e\right)}{G},
	\]
	from which we deduce that 
	\begin{align}
		\CExp{\mathcal{F}^\alpha_G\left(D^\alpha \pi e\right)}{G} = 1. \label{eq:condexprank}
	\end{align}
	The result now follows by inserting~\eqref{eq:condexprank} in~\eqref{eq:avuniformranks1}.
	
	ii) Again, let $\uedge'$ be an independent copy of $\uedge$ and $\alpha, \beta \in \{+, -\}$. For 
	$x, y \in \R$, we write $\cont{F}_G^\alpha(x) = F_{\cont{D^\alpha \pi_\ast \uedge}|G}(x)$ and 
	similarly $\cont{F}_G^\beta(y) = F_{\cont{D^\beta \pi^\ast \uedge}|G}(y)$. Then we have,
	\begin{align*}
		\frac{R^\alpha \pi_\ast e}{|E|}	&= \frac{1}{|E|} \sum_{f \in E} \Ind{D^\alpha \pi_\ast f + U_f \ge 
			D^\alpha \pi_\ast e + U_e} \\
		&= \frac{1}{|E|} \sum_{f \in E} \Ind{D^\alpha \pi_\ast f + U_f > D^\alpha \pi_\ast e + U_e}
		 + \Ind{f = e} \\
		&= 1 - \CExp{\Ind{D^\alpha \pi_\ast \uedge' + U_{\uedge'} \le D^\alpha \pi_\ast e + U_e}}{G}
		 + \frac{1}{|E|} \\
		&= 1 - \cont{F}_G^\alpha\left(D^\alpha \pi_\ast e + U_e\right) + \frac{1}{|E|} \numberthis 
		\label{eq:avuniformranks_ii_1}.
	\end{align*}
	Using similar calculations we get
	\begin{align}
		\frac{R^\beta \pi^\ast e}{|E|} &= 1 - \cont{F}_G^\beta\left(D^\beta \pi^\ast e + W_e\right) 
			+ \frac{1}{|E|} \label{eq:avuniformranks_ii_2}.
	\end{align}
	Now, using both~\eqref{eq:avuniformranks_ii_1} and~\eqref{eq:avuniformranks_ii_2}, we obtain,
	\begin{align*}
		\frac{1}{|E|} \sum_{e \in E} \frac{R^\alpha \pi_\ast e}{|E|}\frac{R^\beta \pi^\ast e}{|E|}
		&= 1 + \frac{2}{|E|} + \frac{1}{|E|^2} + \frac{1}{|E|}\sum_{e \in E} 
			\cont{F}_G^\alpha\left(D^\alpha \pi_\ast e + U_e\right)\cont{F}_G^\beta\left(D^\beta 
			\pi^\ast e + W_e\right)\\
		&\hspace{10pt}- \left(1 + \frac{1}{|E|}\right)\frac{1}{|E|}\sum_{e \in E} \left(\cont{F}_G^\alpha
			\left(D^\alpha \pi_\ast e + U_e\right) + \cont{F}_G^\beta\left(D^\beta \pi^\ast e + W_e\right)
			\right) \\
		&= 1 + \frac{2}{|E|} + \frac{1}{|E|^2} + \CExp{\cont{F}_G^\alpha\left(\cont{D^\alpha 
			\pi_\ast\uedge}\right)\cont{F}_G^\beta\left(\cont{D^\beta \pi^\ast\uedge}\right)}{G} \\
		&\hspace{10pt}- \left(1 + \frac{1}{|E|}\right)\left(\CExp{\cont{F}_G^\alpha\left(\cont{D^\alpha 
			\pi_\ast\uedge}\right)}{G} + \CExp{\cont{F}_G^\beta\left(\cont{D^\beta\pi^\ast\uedge}\right)}
			{G}\right) \\
		&= \frac{1}{4}\CExp{\mathcal{F}^\alpha_G\left(D^\alpha \pi_\ast \uedge\right)
				\mathcal{F}^\beta_G\left(D^\beta \pi^\ast \uedge\right)}{G} + \frac{1}{|E|} + \frac{1}{|E|^2}.
	\end{align*}
	The last line follows by first using Propositions~\ref{prop:condcontprop} and~\ref{prop:condcontjointprop} 
	to rewrite the conditional expectations and then applying~\eqref{eq:condexprank}.
\end{proof}

%

%% file: randomgraphs_kendallstau.tex
\subsection{Kendall's Tau}

The definition for $\kendall_\alpha^\beta(G)$ is, cf.~\cite{hoorn2013},
\[
	\kendall_\alpha^\beta(G) = \frac{2(\mathcal{N}_C(G) - \mathcal{N}_D(G))}{|E|(|E| - 1)},
\]
where $\mathcal{N}_C(G)$ and $\mathcal{N}_D(G)$ denote the number of concordant and discordant pairs, 
respectively, among $\left(D^\alpha \pi_\ast e, D^\beta \pi^\ast e\right)_{e \in E}$. We recall that a 
pair $\left(D^\alpha \pi_\ast e, D^\beta \pi^\ast e\right)$ and $\left(D^\alpha \pi_\ast f, D^\beta 
\pi^\ast f\right)$, for $e, f \in E$ is called (discordant) concordant if 
\[
	\left(D^\alpha \pi_\ast e - D^\alpha \pi_\ast f\right)\left(D^\beta \pi^\ast e - D^\beta \pi^\ast 
	f\right) \, (< 0) > 0.
\] 
Therefore we have, for the concordant pairs,
\begin{align*}
	\frac{2}{|E|^2} \mathcal{N}_C(G) 
	&= \frac{1}{|E|^2} \sum_{e, f \in E} I\left\{D^\alpha \pi_\ast(f) < D^\alpha \pi_\ast(e), 
		D^\beta \pi^\ast(f) < D^\beta \pi^\ast(e)\right\} \\
	&+ \frac{1}{|E|^2} \sum_{e, f \in E} I\left\{D^\alpha \pi_\ast(f) > D^\alpha \pi_\ast(e), 
		D^\beta \pi^\ast(f) > D^\beta \pi^\ast(e)\right\} \\
	&= \CExp{H^{\alpha, \beta}_{G}\left(D^\alpha \pi_\ast \uedge - 1, D^\beta \pi^\ast \uedge 
		- 1\right)}{G} \\
	&\hspace{10pt}+ 1 - \CExp{F^\alpha_{G}\left(D^\alpha \pi_\ast \uedge\right)}{G}
		- \CExp{F^\beta_{G}\left(D^\beta \pi^\ast \uedge\right)}{G}\\
	&\hspace{10pt}+ \CExp{H^{\alpha, \beta}_{G}\left(D^\alpha \pi_\ast \uedge, 
		D^\beta \pi^\ast \uedge\right)}{G}. \\
\end{align*}
In a similar fashion we get for the discordant pairs
\begin{align*}
	\frac{2}{|E|^2} \mathcal{N}_D(G) &= \CExp{F^\alpha_{G}\left(D^\alpha \pi_\ast \uedge - 
		1\right)}{G} + \CExp{F^\beta_{G}\left(D^\beta \pi^\ast \uedge - 1\right)}{G} \\
	&\hspace{10pt}- \CExp{H^{\alpha, \beta}_{G}\left(D^\alpha \pi_\ast \uedge - 1, 
		D^\beta \pi^\ast \uedge\right)}{G} - \CExp{H^{\alpha, \beta}_{G}\left(D^\alpha \pi_\ast 
		\uedge, D^\beta \pi^\ast \uedge - 1\right)}{G}.
\end{align*}
Combining the above with~\eqref{eq:condexprank} we conclude that
\begin{align}
	\kendall_\alpha^\beta(G) &= \CExp{\mathcal{H}_G^{\alpha, \beta}\left(D^\alpha \pi_\ast \uedge,
	D^\beta \pi^\ast \uedge\right)}{G} - 1 + o_\pr(|E|^{-1}). \label{eq:kendallgraph}
\end{align}

%% file: randomgraphs_convergence.tex
\subsection{Statistical consistency of rank correlations}

We will now prove that the rank correlations defined in the previous two sections are, under natural 
regularity conditions on the degree sequences, consistent statistical estimators.

For a sequence $\{G_n\}_{n \in \N}$ of random graphs with $|V_n| = n$, it is common in the theory of 
random graphs to assume convergence of the empirical degree distributions, see for instance Condition 
7.5 in~\cite{VanDerHofstad2007}, Condition 4.1 in~\cite{chen2013}. Here, similarly to
\cite{Litvak2012}, we impose the following regularity condition on the degrees at the end points of 
edges.

\begin{cond}\label{cond:degreedegree}
Given a sequence $\{G_n\}_{n \in \N}$ of random graphs with $|V_n| = n$ and $\alpha, \beta
\in \{+, -\}$ there exist integer valued random variables $\rdegree^\alpha$ and $\rdegree^\beta$,
not concentrated in a single point, such that 
\[
	\cdlim{D^\alpha_n \pi_\ast \uedge_n, D^\beta_n \pi^\ast \uedge_n}{G_n}{\left(\rdegree^\alpha, 
	\rdegree^\beta\right)} \quad \text{as } n \to \infty,
\]
where $\uedge_n$ is a uniformly sampled edge in $G_n$.
\end{cond}

In the previous two sections it was shown that $\spearman_\alpha^\beta(G)$, \, $\spearmanaverage
_\alpha^\beta(G)$ and $\kendall_\alpha^\beta(G)$ on a random graph $G$ are related to, 
respectively,
\[
	\CExp{\mathcal{F}_G^\alpha\left(D^\alpha \pi_\ast \uedge\right)\mathcal{F}_G^\beta\left(D^\beta
	\pi^\ast \uedge\right)}{G} \, \text{ and } \,
	\CExp{\mathcal{H}_G^{\alpha, \beta}\left(D^\alpha \pi_\ast \uedge, D^\beta \pi^\ast \uedge\right)}
	{G}.
\] 
Note that these are in fact empirical versions of the functions appearing in the definitions of 
Spearman's rho and Kendall's tau, cf. \eqref{eq:spearmanxy} and \eqref{eq:kendallxy}. The following 
result formalizes these observations and states that under Condition~\ref{cond:degreedegree}, 
$\spearman_\alpha^\beta(G_n)$, $\spearmanaverage_\alpha^\beta(G_n)$ and $\kendall_\alpha^\beta(G_n)$ 
are indeed consistent statistical estimators of correlation measures associated with Spearman's rho 
and Kendall's tau.

\begin{thm}\label{thm:convdegdegcorr}
	Let $\alpha, \beta \in \{+,-\}$ and $\{G_n\}_{n \in \N}$ be a sequence of graphs satisfying 
	Condition~\ref{cond:degreedegree} such that as $n \to \infty$, $|E_n| \plim \infty$.
	Then, as $n \to \infty$,
	\begin{enumerate}[\upshape i)]
		\item
		$\spearman_\alpha^\beta(G_n) \plim \spearman\left(\rdegree^\alpha, \rdegree^\beta\right)$,
		\item
		$\displaystyle \spearmanaverage_\alpha^\beta(G_n) \plim \frac{\spearman\left(\rdegree^\alpha, 
			\rdegree^\beta\right)}{3\sqrt{S_{\rdegree^\alpha}\left(\rdegree^\alpha\right)
			S_{\rdegree^\beta}\left(\rdegree^\beta\right)}}$,	\par \smallskip
			where $S_{\rdegree^\alpha}\left(\rdegree^\alpha\right) = \Exp{F_{\rdegree^\alpha}
			\left(\rdegree^\alpha\right)F_{\rdegree^\alpha}\left(\rdegree^\alpha - 1\right)}$,	and
		\item
		$\kendall_\alpha^\beta(G_n) \plim \kendall\left(\rdegree^\alpha, \rdegree^\beta\right)$.
	\end{enumerate}
	Moreover, we have convergence of the first moments:
	\begin{enumerate}[\upshape i)]
		\setcounter{enumi}{3}
		\item 
		$\displaystyle \lim_{n \to \infty} \Exp{\spearman_\alpha^\beta(G_n)} = 
		\spearman\left(\rdegree^\alpha, \rdegree^\beta\right)$,
		\item 
		$\displaystyle \lim_{n \to \infty} \Exp{\spearmanaverage_\alpha^\beta(G_n)} = 
		\frac{\spearman\left(\rdegree^\alpha, \rdegree^\beta\right)}
		{3\sqrt{S_{\rdegree^\alpha}\left(\rdegree^\alpha\right)S_{\rdegree^\beta}
		\left(\rdegree^\beta\right)}}$ \, and
		\item
		$\displaystyle \lim_{n \to \infty} \Exp{\kendall_\alpha^\beta(G_n)} = 
		\kendall\left(\rdegree^\alpha, \rdegree^\beta\right)$.
	\end{enumerate}
\end{thm}

\begin{proof}
	i)	By Proposition~\ref{prop:avuniformranks} we have that
	\begin{align*}
		\frac{12}{|E_n|}\sum_{e \in E_n}\frac{R_n^\alpha \pi_\ast e}{|E_n|}\frac{R_n^\beta \pi^\ast e}{|E_n|}
		&= 3\CExp{\mathcal{F}^\alpha_{G_n}\left(D_n^\alpha \pi_\ast \uedge_n\right)
		\mathcal{F}^\beta{G_n}\left(D_n^\beta \pi^\ast \uedge_n\right)}{G_n} + o_\pr(|E_n|^{-1}).
	\end{align*}
	From this and the fact that $|E_n| \plim \infty$ it follows that,
	\begin{align*}
		\spearman_\alpha^\beta(G_n) &= \frac{1}{1 - |E_n|^{-2}}\left(\frac{12}{|E_n|}
			\sum_{e \in E_n}\frac{R_n^\alpha \pi_\ast e}{|E_n|}\frac{R_n^\beta \pi^\ast e}{|E_n|}  
			- 3\frac{|E_n|(|E_n| + 1)^2}{|E_n|^3}\right)\\
		&= 3\CExp{\mathcal{F}^\alpha_{G_n}\left(D_n^\alpha \pi_\ast \uedge_n\right)
		\mathcal{F}^\beta_{G_n}\left(D_n^\beta \pi^\ast \uedge_n\right)}{G_n} - 3 + o_\pr(|E_n|^{-1}) \\
		&\plim \spearman\left(\rdegree^\alpha, \rdegree^\beta\right) \quad \text{as } n \to \infty,
	\end{align*}
	where the last line follows from Theorem~\ref{thm:convrankcorr}.
		
	ii) From~\eqref{eq:avrankgraph} it follows that,
	\[
		\left(\frac{\overline{R}_n^\alpha \pi e}{|E_n|}\right)^2 = \left(1 + \frac{1}{2|E_n|}\right)^2 
		- \left(1 + \frac{1}{2|E_n|}\right)\mathcal{F}^\alpha_{G_n}\left(D^\alpha \pi e\right) 
		+ \frac{1}{4}\mathcal{F}^\alpha_{G_n}\left(D^\alpha \pi e\right)^2.
	\]
	Therefore,
	\begin{align*}
		\frac{1}{|E_n|}\sum_{e \in E_n} \left(\frac{\overline{R}_n^\alpha \pi e}{|E_n|}\right)^2 
		&= \left(1 + \frac{1}{2|E_n|}\right)^2 + \frac{1}{4}\CExp{\mathcal{F}^\alpha_{G_n}
			\left(D^\alpha \pi \uedge_n\right)^2}{G_n} \\
		&\hspace{10pt}- \left(1 + \frac{1}{2|E_n|}\right)\CExp{\mathcal{F}^\alpha_{G_n}
			\left(D^\alpha \pi \uedge_n\right)}{G_n} \\
		&= 1 + \frac{1}{4}\CExp{\mathcal{F}^\alpha_{G_n}\left(D^\alpha \pi \uedge_n\right)^2}{G_n}
		- \CExp{\mathcal{F}_{G_n}^\alpha\left(D_n^\alpha \pi \uedge_n\right)}{G_n} 
			+ o_\pr(|E_n|^{-1}) \\
		&\plim 1 + \frac{1}{4}\Exp{\mathcal{F}_{\rdegree^\alpha}\left(\rdegree^\alpha\right)^2}
			- \Exp{\mathcal{F}_{\rdegree^\alpha}\left(\rdegree^\alpha\right)} \quad \text{as } 
			n \to \infty \\
		&= \frac{1}{4} + \frac{1}{4}\Exp{F_{\rdegree^\alpha}\left(\rdegree^\alpha\right)
			F_{\rdegree^\alpha}\left(\rdegree^\alpha - 1\right)},
	\end{align*}
	where we used Lemma~\ref{lem:contmoments} for the last line.
	It follows that, as $n \to \infty$,
	\[
		\frac{4}{|E_n|}\sum_{e \in E_n} \left(\frac{\overline{R}_n^\alpha \pi e}{|E_n|}\right)^2
		- \frac{|E_n|(|E_n| + 1)^2}{|E_n|^3} 
		\plim \Exp{F^\alpha\left(\rdegree^\alpha\right)F^\alpha\left(\rdegree^\alpha - 1\right)}.
	\]
	Since $\rdegree^\alpha$ and $\rdegree^\beta$ are not concentrated in one point the above term is 
	non-zero. Now,	combining this with	Proposition~\ref{prop:avuniformranks} 
	i) and applying Theorem~\ref{thm:convrankcorr}, we obtain
	\[
		\spearmanaverage_\alpha^\beta(G_n) \plim \frac{\spearman\left(\rdegree^\alpha, \rdegree^\beta\right)}
		{3\sqrt{S_{\rdegree^\alpha}\left(\rdegree^\alpha\right)S_{\rdegree^\beta}\left(\rdegree^\beta\right)}}
		\quad \text{as } n \to \infty.
	\]
		
	iii) Combining~\eqref{eq:kendallgraph} with Theorem~\ref{thm:convrankcorr} yields, as $n \to \infty$,
	\begin{align*}
		\kendall_\alpha^\beta(G_n) = \CExp{\mathcal{H}_{G_n}^{\alpha, \beta}\left(D_n^\alpha \pi_\ast \uedge_n,
			D_n^\beta \pi^\ast \uedge_n\right)}{G_n} - 1 + o_\pr(|E_n|^{-1}) \plim \kendall\left(\rdegree^\alpha, 
			\rdegree^\beta\right).
	\end{align*}
	Finally, iv),v),vi) now follow from, respectively, i), ii) and iii) since $\spearman_\alpha^\beta(G_n)$,
	$\spearmanaverage_\alpha^\beta(G_n)$ and $\kendall_\alpha^\beta(G_n)$ are bounded.
\end{proof}

Comparing results i) and iv) to ii) and v), note that the way in which ties are resolved influences
the measure estimated by Spearman's rho on random directed graphs. In particular, resolving ties
uniformly at random yields the value corresponding to Spearman's rho for the two limiting integer 
valued random variables $\rdegree^\alpha$ and $\rdegree^\beta$ as defined in \cite{Mesfioui2005}, 
in the infinite size network limit.

%% file: configurationmodel.tex
\section{Directed Configuration Model}\label{sec:config}

In this section we will analyze degree-degree dependencies for the directed \emph{Configuration Model}
(CM), as described and analyzed in~\cite{chen2013}. First, in Section~\ref{ssec:multi-graphs}, we 
analyze the model where in- and out-links are connected at random, which, in general, results in a 
multi-graph. Then we move on to two other models that produce simple graphs: the \emph{Repeated} and 
\emph{Erased} Configuration Model (RCM and ECM). By applying Theorem \ref{thm:convdegdegcorr}, in 
Sections~\ref{ssec:repeated}~and~\ref{ssec:erased}, we will show that RCM and ECM can be used as null 
models for the rank correlations $\spearman$, $\spearmanaverage$ and $\kendall$.

\input{configurationmodel_multigraphs}

\input{configurationmodel_repeated}

\input{configurationmodel_erased}

%% file: configurationmodel_multigraphs.tex
\subsection{General model: multi-graphs}
\label{ssec:multi-graphs}

The directed Configuration Model in~\cite{chen2013} starts with picking two target distributions 
$F_-$, $F_+$ for the in- and out-degrees, respectively, stochastically bounded from above by regularly 
varying distributions. We will adopt notations from~\cite{chen2013} and let $\gamma$ and $\xi$ denote 
random variables with distributions $F_-$ and $F_+$, respectively. It is assumed that $\Exp{\gamma} = 
\Exp{\xi} < \infty$. The next step is generating a bi-degree sequence of inbound and outbound stubs. 
This is done by first taking two independent sequences of $n$ independent copies of $\gamma$ and $\xi$, 
which are then modified into a sequence of in- and outbound stubs
\[
	\bisqn{G} = \left(\widehat{D}^+(v), \widehat{D}^-(v)\right)_{v \in V},
\]
using the algorithm in \cite{chen2013}, Section 2.1. This algorithm ensures that the total number of 
in- and outbound stubs is the same, $|\widehat{E}| = \sum_{v \in V} \widehat{D}^\alpha(v)$, $\alpha \in 
\{+, -\}$. Using this bi-degree sequence, a graph is build by randomly pairing the 
stubs to form edges. We call a graph generated by this model a \emph{Configuration Model graph}, or CM 
graph for short. We remark that a CM graph in general does not need to be simple.

Given a vertex set $V$, a bi-degree sequence $\bisqn{G}$ and $v \in V$, we denote by 
$v^+_i$, $v^-_j$ for $1 \le i \le \widehat{D}^+(v)$ and $1 \le j \le \widehat{D}^-(v)$, respectively, 
the outbound and inbound stubs of $v$. For $v, w \in V$, we denote by $\{v_i^+ \to w_j^-\}$ the event 
that the outbound stub $v_i^+$ is connected to the inbound stub $w_j^-$ and by $\{v_i^+ \to w\}$ the 
event that $v_i^+$ is connected to an inbound stub of $w$. By definition of CM, it follows that 
$\Prob{v_i^+ \to w_j^- | \bisqn{G}} = 1/|\widehat{E}|$ and hence 
$\Prob{v_i^+ \to w | \bisqn{G}} = \widehat{D}^-(w)/|\widehat{E}|$. Furthermore we observe that 
$|\widehat{E}_n(e)| = \sum_{i = 1}^{\widehat{D}_n^+ \pi_\ast e} \Ind{(\pi_\ast e)^+_i \to \pi^\ast e}$. Given a 
random graph $G$, we denote
\[
	I_e^{\alpha, \beta}(k, l) = \Ind{D^\alpha \pi_\ast e = k}\Ind{D^\beta \pi^\ast e = l},
\]
where $\alpha, \beta \in \{+, -\}$, $k, l \in \N$ and $e \in V^2$.

For proper reference we summarize some results from Proposition 2.5, in~\cite{chen2013}, which we
will use in the remainder of this paper.

\begin{prop}[\cite{chen2013}, Proposition 2.5]\label{prop:bisqn_chen}
	Let $\bisqn{G_n}$ be the bi-degree sequence on $n$ vertices, as generated in Section 2.1 of
	\cite{chen2013}, and $k, l \in \N$. Then, as $n \to \infty$,
	\begin{align*}
		&\frac{1}{n} \sum_{v \in V_n} \Ind{\widehat{D}_n^+ v = k}\Ind{\widehat{D}_n^- v = l} \plim 
			\Prob{\xi = k}\Prob{\gamma = l}, \\
		&\frac{1}{n} \sum_{v \in V_n} \widehat{D}_n^+ v \plim \Exp{\xi} \quad \text{and} \quad 
		\frac{1}{n} \sum_{v \in V_n} \widehat{D}_n^- v \plim \Exp{\gamma}.
	\end{align*}
\end{prop}

Given a random graph $G = (V, E)$, we will use $\degreesqn{G}$ as a short hand notation for its 
degree sequence $(D^-(v),D^+(v))_{v\in V}$. We emphasize that for a graph generated using an 
initial bi-degree sequence, the eventual degree sequence $\degreesqn{G}$ can be different from 
$\bisqn{G}$. This, for example, is true for the ECM, Section~\ref{ssec:erased}, where, after the 
random pairing of the stubs, self-loops are removed and multiple edges are merged. 

In order to apply Theorem~\ref{thm:convdegdegcorr} to a sequence of (multi-)graphs $\{G_n\}_{n \in 
\N}$ generated by CM, we need to prove that
\[
	\cdlim{D^\alpha_n \pi_\ast \uedge_n, D_n^\beta \pi^\ast \uedge_n}{G_n}
	{\left(\rdegree^\alpha, \rdegree^\beta\right)},
\] 
for some integer valued random variables $\rdegree^\alpha$ and $\rdegree^\beta$. For this, it 
suffices to show that, as $n \to \infty$,
\[
	H^{\alpha, \beta}_{G_n}(k, l) \plim 
	H_{\rdegree^\alpha, \rdegree^\beta}(k, l),
\]
for all $k, l \in \N$. We will prove this by showing that
\[
	\CExp{I_{\uedge_n}^{\alpha, \beta}(k, l)}{G_n} \plim \Prob{\rdegree^\alpha = k, \rdegree^\beta = l},
\]
as $n \to \infty$, using a second moment argument as follows. Given a sequence $\{G_n\}_{n \in \N}$ of 
graphs, $\alpha, \beta \in \{+, -\}$ and $k, l \in \N$, we will 
show that the empirical joint probability $\Exp{\CExp{I_{\uedge_n}^{\alpha, \beta}(k, l)}{G_n}}$ 
converges to $\Prob{\rdegree^\alpha = k, \rdegree^\beta = l}$. Then we will prove that the variance of
$\CExp{I_{\uedge_n}^{\alpha, \beta}(k, l)}{G_n}$ converges to zero. 

We start with expressing the first and second moment of $\CExp{I_{\uedge_n}^{\alpha, \beta}(k, l)}
{G_n}$, for CM graphs, conditioned on the bi-degree sequence $\bisqn{G_n}$ in terms of 
the degrees. We observe that, for $\alpha, \beta \in \{+, -\}$, $e \in V_n^2$ and $k, l \in \N$, the events 
$\left\{D_n^\alpha \pi_\ast e = k \right\}$ and $\left\{D^\beta_n \pi^\ast e = l \right\}$ are completely 
defined by $\bisqn{G_n}$, hence so is $I_e^{\alpha, \beta}(k, l)$.  We remark that, since CM leaves the 
number of inbound and outbound stubs intact, we have $\degreesqn{G_n} = \bisqn{G_n}$. However, in this 
section we will keep using hats, e.g. $\widehat{D}_n$ instead of $D_n$, to emphasize that $G_n$ can be a 
multi-graph. 

\begin{lem}\label{lem:rcmedgecount}
	Let $\{G_n\}_{n \in \N}$ be a sequence of CM graphs with $|V_n| = n$ and $\alpha, \beta \in \{+, -\}$. 
	Then, for each $k, l \in \N$,
	\begin{enumerate}[\upshape i)]
		\item 
		$\displaystyle 
		\CExp{\CExp{I_{\uedge_n}^{\alpha, \beta}(k, l)}{G_n}}{\bisqn{G_n}} 
			= \sum_{e \in V_n^2} I_e^{\alpha, \beta}(k, l) \frac{\widehat{D}_n^+ \pi_\ast e \, \widehat{D}_n^- 
			\pi^\ast e}{|\widehat{E}_n|^2}$ \, and
		\item
		$\displaystyle 
			\CExp{\CExp{I_{\uedge_n}^{\alpha, \beta}(k, l)}{G_n}^2}{\bisqn{G_n}}
			= \left(\sum_{e \in V_n^2} I_{e}^{\alpha, \beta}(k, l)\frac{\widehat{D}_n^+ \pi_\ast e 
			\widehat{D}_n^- \pi^\ast e}{|\widehat{E}_n|^2}\right)^2 + o_\pr(1).$
	\end{enumerate}
\end{lem}

\begin{proof} \hfill
	\begin{align*}
	\text{i) } \CExp{\CExp{I_{\uedge_n}^{\alpha, \beta}(k, l)}{G_n}}{\bisqn{G_n}} 
		&= \CExp{\sum_{e \in V_n^2} I_e^{\alpha, \beta}(k, l) \frac{|\widehat{E}_n(e)|}{|\widehat{E}_n|}}{\bisqn{G_n}}
		\numberthis \label{eq:cm_firstmoment}\\
		&= \frac{1}{|\widehat{E}_n|}\sum_{e \in V_n^2} I_e^{\alpha, \beta}(k, l) \CExp{|\widehat{E}_n(e)|}{\bisqn{G_n}} \\
		&= \frac{1}{|\widehat{E}_n|}\sum_{e \in V_n^2} I_{e}^{\alpha, \beta}(k, l)\CExp{\sum_{i = 1}^{
			\widehat{D}_n^+ \pi_\ast e} \Ind{(\pi_\ast e)^+_i \to \pi^\ast e}}{\bisqn{G_n}} \\
		&= \sum_{e \in V_n^2} I_{e}^{\alpha, \beta}(k, l)\frac{\left(\widehat{D}_n^+ \pi_\ast e\right)
			\left(\widehat{D}_n^- \pi^\ast e\right)}{|\widehat{E}_n|^2}.
	\end{align*}
	
	ii)	Following similar calculations as above we get,
	 \begin{align*}
		&\CExp{\CExp{I_{\uedge_n}^{\alpha, \beta}(k, l)}{G_n}^2}{\bisqn{G_n}} \\
		&= \CExp{\sum_{e, f \in V_n^2} 
			I_{e}^{\alpha, \beta}(k, l)	I_{f}^{\alpha, \beta}(k, l) \frac{|\widehat{E}_n(e)| \, |\widehat{E}_n(f)|}
			{|\widehat{E}_n|^2}}{\bisqn{G_n}} \numberthis \label{eq:cm_secondmoment}\\
		\displaybreak \\
		&= \frac{1}{|\widehat{E}_n|^2} \sum_{e, f \in V_n^2} \hspace{-4pt}
			\left( \vphantom{\sum_{i = 1}^{D_n^+ \pi_\ast e}}	I_{e}^{\alpha, \beta}
			(k, l)I_{f}^{\alpha, \beta}(k, l) \right. \\
		&\hspace{75pt}\left.\sum_{i = 1}^{\widehat{D}_n^+ \pi_\ast e}\sum_{s = 1}^{\widehat{D}_n^+ 
			\pi_\ast f}\CExp{\Ind{(\pi_\ast e)^+_i \to \pi^\ast e}	\Ind{(\pi_\ast f)^+_s \to 
			\pi^\ast f}}{\bisqn{G_n}}\right). \numberthis \label{eq:rcmedgecount_ii_main}
	\end{align*}
	We will, for $e, f \in V_n^2$, analyze 
	\begin{align}
		\frac{1}{|\widehat{E}_n|^2}\sum_{i = 1}^{\widehat{D}_n^+ \pi_\ast e}\sum_{s = 1}^{\widehat{D}_n^+ 
		\pi_\ast f}\CExp{\Ind{(\pi_\ast e)^+_i \to \pi^\ast e}\Ind{(\pi_\ast f)^+_s \to \pi^\ast f}}
		{\bisqn{G_n}}	\label{eq:rcmedgecount1}
	\end{align}
	for all different cases, $e = f$, $e \cap f = \emptyset$, $e_\ast = f_\ast$ and $e^\ast = f^\ast$.
	First, suppose that $e = f$. Then~\eqref{eq:rcmedgecount1} equals
	\begin{align*}
		&\frac{1}{|\widehat{E}_n|^2}\sum_{i, s = 1}^{\widehat{D}_n^+ \pi_\ast e} 
			\sum_{j, t = 1}^{\widehat{D}_n^- \pi^\ast e} \frac{\Ind{i = s}\Ind{j = t}}{|\widehat{E}_n|} 
			+ \frac{\Ind{i \ne s}\Ind{j \ne t}}{|\widehat{E}_n|(|\widehat{E}_n| - 1)}.
	\end{align*}
	Writing out the sums and using that $e = f$ we obtain,
	\begin{align}
		\eqref{eq:rcmedgecount1} &= \frac{\widehat{D}_n^+ \pi_\ast e \widehat{D}_n^- \pi^\ast e 
			\widehat{D}_n^+ \pi_\ast f \widehat{D}_n^- \pi^\ast f}{|\widehat{E}_n|^3(|\widehat{E}_n| - 1)} 
			\label{eq:rcmedgecount_ii_1}\\
		&\hspace{10pt}+ \frac{\left(\widehat{D}_n^+ \pi_\ast e\right)
			\left(\widehat{D}_n^- \pi^\ast e\right)}{|\widehat{E}_n|^3} 
			+ \frac{\left(\widehat{D}_n^+ \pi_\ast e\right)
			\left(\widehat{D}_n^- \pi^\ast e\right)}{|\widehat{E}_n|^3(|\widehat{E}_n| - 1)} 
			\label{eq:rcmedgecount_ii_2}\\
		&\hspace{10pt}- \frac{\left(\widehat{D}_n^- \pi^\ast e\right)^2 
			\left(\widehat{D}_n^+ \pi_\ast e\right)}
			{|\widehat{E}_n|^3(|\widehat{E}_n| - 1)} - \frac{\left(\widehat{D}_n^+ \pi_\ast e\right)^2 
			\left(\widehat{D}_n^- \pi^\ast e\right)}
			{|\widehat{E}_n|^3(|\widehat{E}_n| - 1)} \label{eq:rcmedgecount_ii_3}
	\end{align}
	Since for all $k \ge 0$ and $\kappa \in \{+, -\}$ it holds that
	\[
		\frac{1}{|\widehat{E}_n|^{k + 1}}\sum_{v \in V_n} \left(\widehat{D}_n^\kappa v\right)^{k} \le 
		\frac{1}{|\widehat{E}_n|^{k + 1}}\left(\sum_{v \in V_n} \widehat{D}_n^\kappa v\right)^k 
		= \frac{1}{|\widehat{E}_n|},
	\]
	we deduce that the terms in~\eqref{eq:rcmedgecount_ii_2} and~\eqref{eq:rcmedgecount_ii_3}
	contribute as $o_\pr(1)$ in~\eqref{eq:rcmedgecount_ii_main}, from which the result for $e = f$ 
	follows. The calculations for the other three cases for $e, f \in V_n^2$ are similar and are hence 
	omitted. 
\end{proof}

As a direct consequence we have the following

\begin{prop}\label{prop:cm_zerovariance}
	Let $\{G_n\}_{n \in \N}$ be a sequence of CM graphs with $|V_n| = n$ and $\alpha, \beta \in \{+, -\}$. 
	Then, for each $k, l \in \N$, as $n \to \infty$,
	\[
		\left|\CExp{\CExp{I_{\uedge_n}^{\alpha, \beta}(k, l)}{G_n}^2}{\bisqn{G_n}} - 
		 \CExp{\CExp{I_{\uedge_n}^{\alpha, \beta}(k, l)}{G_n}}{\bisqn{G_n}}^2\right| \plim 0.
	\]
\end{prop}

Now, using the convergence results from~\cite{chen2013}, summarized in Proposition~\ref{prop:bisqn_chen},
 we are able to determine the limiting random variables $\rdegree^\alpha$ and $\rdegree^\beta$.

\begin{prop}\label{prop:cm_empirical_conv}
	Let $\{G_n\}_{n \in \N}$ be a sequence of CM graphs with $|V_n| = n$ and $\alpha, 
	\beta \in \{+, -\}$. Then there exist integer valued random variables $\rdegree^\alpha$ and 
	$\rdegree^\beta$ such that for each $k, l \in \N$, as $n \to \infty$,
	\[
		\CExp{\CExp{I_\uedge^{\alpha, \beta}(k, l)}{G_n}}{\bisqn{G_n}} \plim \Prob{\rdegree^\alpha = k}
		\Prob{\rdegree^\beta = l}.
	\]
\end{prop}

\begin{proof}
	First let $(\alpha, \beta) = (+, -)$. Then it follows from Lemma~\ref{lem:rcmedgecount} i) that
	\begin{align*}
		\CExp{\CExp{I_{\uedge_n}^{\alpha, \beta}(k, l)}{G_n}}{\bisqn{G_n}} 
		&= \sum_{v, w \in V_n} \Ind{\widehat{D}_n^+ v = k}\Ind{\widehat{D}_n^- w = l}
				\frac{\widehat{D}^+_n v \widehat{D}^-_n w}{|\widehat{E}_n|^2} \\
		&= \left(\sum_{v \in V_n} \Ind{\widehat{D}_n^+ v = k} \frac{\widehat{D}^+_n v}{|\widehat{E}_n|}
			\right)\left(\sum_{w \in V_n} \Ind{\widehat{D}_n^- w = l} \frac{\widehat{D}^-_n w}
			{|\widehat{E}_n|}\right) \\
		&= \left(k\sum_{v \in V_n} \frac{\Ind{\widehat{D}_n^+ v = k}}{|\widehat{E}_n|}
			\right)\left(l\sum_{w \in V_n} \frac{\Ind{\widehat{D}_n^- w = l}}{|\widehat{E}_n|}\right) \\
		&\plim \frac{k\Prob{\xi = k}}{\Exp{\xi}}\frac{l \, \Prob{\gamma = l}}{\Exp{\gamma}}
		\quad \text{as } \,  n \to \infty,
	\end{align*}
	where the convergence in the last line is by Proposition~\ref{prop:bisqn_chen}. The other three 
	cases are slightly more involved. Consider, for example, $(\alpha, \beta) = (-, +)$. Then we have,
	\begin{align}
		\CExp{\CExp{I_{\uedge_n}^{\alpha, \beta}(k, l)}{G_n}}{\bisqn{G_n}} &= \sum_{v \in V_n} \Ind{\widehat{D}_n^- v = k} 
		\frac{\widehat{D}^+_n v}{|\widehat{E}_n|}
			\sum_{w \in V_n} \Ind{\widehat{D}_n^+ w = l} \frac{\widehat{D}^-_n w}{|\widehat{E}_n|} \label{eq:rcmempiricalconv1}
	\end{align}
	We will first analyze the last summation. 
	\begin{align*}
		\frac{1}{|\widehat{E}_n|}\sum_{w \in V_n} \widehat{D}^-_n (w) \Ind{\widehat{D}_n^+ w = l}
		&= \frac{1}{|\widehat{E}_n|}\sum_{i \in \N} i \sum_{w \in V_n} \Ind{\widehat{D}_n^- w = i}
		\Ind{\widehat{D}_n^+ w = l} \\
		&\plim \frac{\Prob{\xi = l}}{\Exp{\xi}} \sum_{i \in \N} i \Prob{\gamma = i} 
		\quad \text{as } n \to \infty\\
		&= \frac{\Prob{\xi = l} \Exp{\gamma}}{\Exp{\xi}} = \Prob{\xi = l}, \numberthis 
		\label{eq:rcmempiricalconv2}
	\end{align*}
	where we again used Proposition~\ref{prop:bisqn_chen} and $\Exp{\gamma}	= \Exp{\xi}$.	In a 
	similar way we obtain that, as $n \to \infty$,
	\begin{align}
		\frac{1}{|\widehat{E}_n|} \sum_{v \in V_n} \widehat{D}_n^+(v)\Ind{\widehat{D}_n^-(v) = k}
		\plim \Prob{\gamma = k}. \label{eq:rcmempiricalconv3}
	\end{align}
	Applying~\eqref{eq:rcmempiricalconv2} and~\eqref{eq:rcmempiricalconv3} to~\eqref{eq:rcmempiricalconv1}
	we get 
	\[
		\CExp{\CExp{I_{\uedge_n}^{-, +}(k, l)}{G_n}}{\bisqn{G_n}} \plim \Prob{\gamma = k} \Prob{\xi = l}.
	\]
	For the other two cases we obtain, as $n \to \infty$,
	\begin{align*}
		&\CExp{\CExp{I_{\uedge_n}^{+, +}(k, l)}{G_n}}{\bisqn{G_n}}
			\plim \frac{k \Prob{\xi = k} \Prob{\xi = l}}{\Exp{\xi}} \\
		&\CExp{\CExp{I_{\uedge_n}^{-, -}(k, l)}{G_n}}{\bisqn{G_n}}
			\plim \frac{l \Prob{\gamma = k} \Prob{\gamma = l}}{\Exp{\gamma}}
	\end{align*}
	The results now holds if we define $\rdegree^\alpha$ and $\rdegree^\beta$ by their probabilities
	summarized in Table~\ref{tab:limiting_degrees}.
	\begin{table}
	\centerline{
			\begin{tabular}{cccc}
			$\alpha$ & $\beta$ & $\Prob{\rdegree^\alpha = k}$ & $\Prob{\rdegree^\beta = l}$ \\
			\hline
			$+$ & $-$ & $k\Prob{\xi = k}/\Exp{\xi}$ & $l \, \Prob{\gamma = l}/\Exp{\gamma}$ \\
			$-$ & $+$ & $\Prob{\gamma = k}$ & $\Prob{\xi = l}$ \\
			$+$ & $+$ & $k\Prob{\xi = k}/\Exp{\xi}$ & $\Prob{\xi = l}$ \\
			$-$ & $-$ & $\Prob{\gamma = k}$ & $l \, \Prob{\gamma = l}/\Exp{\gamma}$.
		\end{tabular}}
		\caption{Distributions of $\rdegree^\alpha$ and $\rdegree^\beta$ for $\alpha,\beta\in \{+,-\}$.}
		\label{tab:limiting_degrees}
		\end{table}
\end{proof}

We end this section with a convergence result for first and second moment of $\CExp{I_{\uedge_n}^{
\alpha, \beta}(k, l)}{G_n}$.

\begin{prop}\label{prop:cm_conv_empirical_propabilities}
Let $\{G_n\}_{n \in \N}$ be a sequence of CM graphs with $|V_n| = n$ and $\alpha, \beta \in \{+, -\}$. 
Then, for each $k, l \in \N$,
\begin{enumerate}[\upshape i)]
	\item $\displaystyle \lim_{n \to \infty} \Exp{\CExp{I_{\uedge_n}^{\alpha, \beta}(k, l)}{G_n}} = 
		\Prob{\rdegree^\alpha = k}\Prob{\rdegree^\beta = l}$,
	\item $\displaystyle \lim_{n \to \infty} \Exp{\CExp{I_{\uedge_n}^{\alpha, \beta}(k, l)}{G_n}^2} = 
		\Prob{\rdegree^\alpha = k}^2\Prob{\rdegree^\beta = l}^2$,
\end{enumerate}	
and hence, as $n \to \infty$, \, $\displaystyle	\CExp{I_{\uedge_n}^{\alpha, \beta}(k, l)}{G_n} 
\plim \Prob{\rdegree^\alpha = k}\Prob{\rdegree^\beta = l}.$
\end{prop}

\begin{proof} \hfill \par
	\begin{enumerate}[\upshape i)]
		\item Let $k, l \in \N$, then, since
		\begin{align}
			\CExp{\CExp{I_\uedge^{\alpha, \beta}(k, l)}{G_n}}{\bisqn{G_n}} \le 1 \label{eq:cm_conv_emp_1},
		\end{align}
		it follows, using Proposition~\ref{prop:cm_empirical_conv} and dominated 
		convergence, that for each pair $\alpha, \beta \in \{+, -\}$, we have
		\begin{align*}
			\lim_{n\to\infty}\Exp{\CExp{I_{\uedge_n}^{\alpha, \beta}(k, l)}{G_n}}=\Prob{\rdegree^\alpha = k}
			\Prob{\rdegree^\beta = l},
		\end{align*}
		where $\rdegree^\alpha$, $\rdegree^\beta$ have distributions defined in Table
		\ref{tab:limiting_degrees}.
		\item For the second moment we get, using conditioning on $\bisqn{G_n}$,
		\begin{align}
			\lim_{n \to \infty} \Exp{\CExp{I_{\uedge_n}^{\alpha, \beta}(k, l)}{G_n}^2} &= \lim_{n \to \infty}
			\Exp{\CExp{\CExp{I_{\uedge_n}^{\alpha, \beta}(k, l)}{G_n}^2}{\bisqn{G_n}}}\nonumber\\		
			&= \lim_{n \to \infty} \Exp{\left(\sum_{e \in V_n^2} I_{e}^{\alpha, \beta}(k, l)\frac{\widehat{D}_n^+ \pi_\ast e 
				\widehat{D}_n^- \pi^\ast e}{|\widehat{E}_n|^2}\right)^2 + o_{\pr}(1)}  \label{eq:th5-3-1}\\
			&= \lim_{n \to \infty} \Exp{\CExp{\CExp{I_\uedge^{\alpha, \beta}(k, l)}{G_n}}{\bisqn{G_n}}^2+o_{\pr}(1)}
				\label{eq:th5-3-2}\\		
			&= \left(\Prob{\rdegree^\alpha = k}\Prob{\rdegree^\beta = l}\right)^2.\label{eq:th5-3-3}
		\end{align}
		Here \eqref{eq:th5-3-1} follows from Lemma~\ref{lem:rcmedgecount} ii), \eqref{eq:th5-3-2} is by 
		Lemma~\ref{lem:rcmedgecount} i), and~\eqref{eq:th5-3-3} is due to Proposition~\ref{prop:cm_empirical_conv}, 
		continuous mapping theorem,~\eqref{eq:cm_conv_emp_1} and the
		fact that the $o_\pr(1)$ terms are uniformly bounded, see proof Lemma~\ref{lem:rcmedgecount}. The 
		distributions of $\rdegree^\alpha$,	$\rdegree^\beta$ are again given in Table~\ref{tab:limiting_degrees}.
	\end{enumerate}
	The last result now follows by a second moment argument.
\end{proof}

%% file: configurationmodel_repeated.tex
\subsection{Repeated Configuration Model}
\label{ssec:repeated}

Described in Section 4.1 of~\cite{chen2013}, RCM connects
inbound and outbound stubs uniformly at random and then the resulting graph is checked to be simple. 
If not, one repeats the connection step until the resulting graph is simple. If the distributions $F_-$ and 
$F_+$ have finite variances, then the probability of the graph being simple converges to a non-zero 
number, see~\cite{chen2013}, Theorem 4.3. Therefore, throughout this section, we 
will assume that $\Exp{\gamma^2}$, $\Exp{\xi^2} < \infty$. 

Let $\{G_n\}_{n \in \N}$ be again a sequence of CM graphs, and let $S_n$ denote the event that $G_n$ is simple. 
We will prove, in Theorem~\ref{thm:rcmcorrconv} below, that for a sequence of RCM graphs 
of growing size, our three rank correlation measures converge to zero, by showing that for all $\alpha,
\beta \in \{+, -\}$ and $k, l \in \N$,
\[
	\CExp{I_{\uedge_n}^{\alpha, \beta}(k, l)}{G_n, S_n} \plim \Prob{\rdegree^\alpha = k}
	\Prob{\rdegree^\beta = l},
\]
as $n \to \infty$, where $\rdegree^\alpha$ and $\rdegree^\beta$ are random variables whose distributions are 
defined in Table~\ref{tab:limiting_degrees}.

First we show that, asymptotically, conditioning on the graph being simple does not effect the 
conditional expectation $\CExp{\CExp{I_{\uedge_n}^{\alpha, \beta}(k, l)}{G_n}}{\bisqn{G_n}}$.

\begin{lem}\label{lem:rcm_conv_cond_probability}
	Let $\{G_n\}_{n \in \N}$ be a sequence of CM graphs with $|V_n| = n$ and $\alpha, \beta \in \{+, -\}$ 
	and denote by $S_n$ the event that $G_n$ is simple. Then, for each $k, l \in \N$, as $n 
	\to \infty$,
	\[
			\left|\CExp{\CExp{I_{\uedge_n}^{\alpha, \beta}(k, l)}{G_n, S_n}}{\bisqn{G_n}} 
			- \CExp{\CExp{I_{\uedge_n}^{\alpha, \beta}(k, l)}{G_n}}{\bisqn{G_n}}\right| \plim 0.
	\]
\end{lem}

\begin{proof} First, we write
	\begin{align}
		&\left|\CExp{\CExp{I_{\uedge_n}^{\alpha, \beta}(k, l)}{G_n, S_n}}{\bisqn{G_n}} 
			- \CExp{\CExp{I_{\uedge_n}^{\alpha, \beta}(k, l)}{G_n}}{\bisqn{G_n}}\right| \notag \\
		&= \left|\CExp{\CExp{I_{\uedge_n}^{\alpha, \beta}(k, l)}{G_n}\left(\frac{\Ind{S_n}}
			{\Prob{S_n}} - 1\right)}{\bisqn{G_n}}\right|. \label{eq:lem:rcm_conv_cond_probability}
	\end{align}
			
	Next, denote by 
	\[
		\text{Var}\left(\left.\CExp{I_{\uedge_n}^{\alpha, \beta}(k, l)}{G_n}\right| \bisqn{G_n}
	\right) \quad \text{and} \quad \text{Var}\left(\left.\Ind{S_n}\right| \bisqn{G_n}\right)
	\]
	the variance of, respectively	$\CExp{I_{\uedge_n}^{\alpha, \beta}(k, l)}{G_n}$ and $\Ind{S_n}$, 
	conditioned on $\bisqn{G_n}$. Then, by adding and subtracting in~\eqref{eq:lem:rcm_conv_cond_probability} 
	the product of the conditional expectations 
	\[
		\CExp{\CExp{I_{\uedge_n}^{\alpha, \beta}(k, l)}{G_n}}{\bisqn{G_n}}\left(\frac{
			\Prob{S_n|\bisqn{G_n}}}{\Prob{S_n}} - 1\right),
	\]
	we get
	\begin{align*}
		\eqref{eq:lem:rcm_conv_cond_probability}&\le \frac{1}{\Prob{S_n}}\sqrt{\text{Var}\left(\left.\Ind{S_n}\right| 
		\bisqn{G_n}\right)} \sqrt{
			\text{Var}\left(\left.\CExp{I_{\uedge_n}^{\alpha, \beta}(k, l)}{G_n}\right| \bisqn{G_n}\right)} \\
		&\hspace{10pt}+ \left|\CExp{\CExp{I_{\uedge_n}^{\alpha, \beta}(k, l)}{G_n}}{\bisqn{G_n}}\left(\frac{
			\Prob{S_n|\bisqn{G_n}}}{\Prob{S_n}} - 1\right)\right| \\
		&\le \frac{1}{\Prob{S_n}} \sqrt{\text{Var}\left(\left.\CExp{I_{\uedge_n}^{\alpha, \beta}(k, l)}{G_n}\right| 
			\bisqn{G_n}\right)} + \left|\frac{\Prob{S_n|\bisqn{G_n}}}{\Prob{S_n}} - 1\right|.
			\numberthis \label{eq:rcm_edgecount_1}
	\end{align*}
	Following the argument in the first part of the proof of Proposition 4.4 from~\cite{chen2013} we 
	conclude that, $\Prob{\mathcal{S}_n|\bisqn{G_n}}$ and $\Prob{\mathcal{S}_n}$ 
	converge to the same positive limit, hence the latter expression in~\eqref{eq:rcm_edgecount_1}
	is $o_\pr(1)$. The result now follows, since by Proposition~\ref{prop:cm_zerovariance}
	\[
		\text{Var}\left(\left.\CExp{I_{\uedge_n}^{\alpha, \beta}(k, l)}{G_n}\right| \bisqn{G_n}\right) 
		= o_\pr(1).
	\]
\end{proof}

In the next 
theorem we show that the conditions of Theorem~\ref{thm:convdegdegcorr} hold for a sequence of 
RCM graphs, and thus obtain the desired convergence of the three rank correlations, using a second moment argument.

\begin{thm}\label{thm:rcmcorrconv}
	Let $\{G_n\}_{n \in \N}$ be a sequence of RCM graphs with $|V_n| = n$ and $\alpha, \beta \in \{+, -\}$. 
	Then, as $n \to \infty$,
	\[
		\spearman_\alpha^\beta(G_n) \plim 0, \quad \spearmanaverage_\alpha^\beta(G_n) \plim 0 \quad 
		\text{and} \quad	\kendall_\alpha^\beta(G_n) \plim 0.
	\]
\end{thm}

\begin{proof}
	Instead of conditioning on RCM graphs we condition on CM graphs $G_n$ and the event that it is simple, $S_n$.
	Let $k, l \in \N$ and let $\rdegree^\alpha$, $\rdegree^\beta$ have distributions defined in Table
	\ref{tab:limiting_degrees}. Then, for each pair $\alpha, \beta \in \{+, -\}$, we have
	\begin{align*}
		&\left|\CExp{\CExp{I_{\uedge_n}^{\alpha, \beta}(k, l)}{G_n, S_n}}{\bisqn{G_n}} - \Prob{
			\rdegree^\alpha = k}\Prob{\rdegree^\beta = l}\right| \\
		&\le \left|\CExp{\CExp{I_{\uedge_n}^{\alpha, \beta}(k, l)}{G_n, S_n}}{\bisqn{G_n}} - \CExp{\CExp{
			I_{\uedge_n}^{\alpha, \beta}(k, l)}{G_n}}{\bisqn{G_n}}\right| \\
		&\hspace{10pt}+ \left|\CExp{\CExp{I_{\uedge_n}^{
			\alpha, \beta}(k, l)}{G_n}}{\bisqn{G_n}} - \Prob{\rdegree^\alpha = k}
			\Prob{\rdegree^\beta = l}\right|.
	\end{align*}
	Hence by Lemma~\ref{lem:rcm_conv_cond_probability} and Proposition~\ref{prop:cm_empirical_conv} it 
	follows that, as $n \to \infty$,
	\begin{align*}
		\CExp{\CExp{I_{\uedge_n}^{\alpha, \beta}(k, l)}{G_n, S_n}}{\bisqn{G_n}} \plim \Prob{\rdegree^\alpha = k}
			\Prob{\rdegree^\beta = l}.
	\end{align*}
	Since $\CExp{\CExp{I_{\uedge_n}^{\alpha, \beta}(k, l)}{G_n, S_n}}{\bisqn{G_n}} \le 1$,
	dominated convergence and the above imply that
	\begin{align}
		\lim_{n \to \infty} \Exp{\CExp{I_{\uedge_n}^{\alpha, \beta}(k, l)}{G_n, S_n}} 
			= \Prob{\rdegree^\alpha = k}\Prob{\rdegree^\beta = l}. \label{eq:rcm_conv_correlations_1} 
	\end{align}
	
	For the second moment we have
	\begin{align}
		&\left|\CExp{\CExp{I_{\uedge_n}^{\alpha, \beta}(k,l)}{G_n, S_n}^2}{\bisqn{G_n}} - 
			\Prob{\rdegree^\alpha = k}^2\Prob{\rdegree^\beta = l}^2\right| \notag \\
		&\le \left|\CExp{\left(\left(\frac{\Ind{S_n}}{\Prob{S_n}}\right)^2 - 1\right)\CExp{I_{\uedge_n}^{\alpha, 
			\beta}(k,l)}{G_n}^2}{\bisqn{G_n}}\right| \label{eq:rcm_conv_correlations_2}\\
		&\hspace{10pt}+ \left|\CExp{\CExp{I_{\uedge_n}^{\alpha, \beta}(k,l)}{G_n}^2}{\bisqn{G_n}} - 
			\CExp{\CExp{I_{\uedge_n}^{\alpha, \beta}(k,l)}{G_n}}{\bisqn{G_n}}^2\right| 
			\label{eq:rcm_conv_correlations_3}\\
		&\hspace{10pt}+ \left|\CExp{\CExp{I_{\uedge_n}^{\alpha, \beta}(k,l)}{G_n}}{\bisqn{G_n}}^2 - 
			\Prob{\rdegree^\alpha = k}^2\Prob{\rdegree^\beta = l}^2\right| \label{eq:rcm_conv_correlations_4}
	\end{align}

	From Proposition~\ref{prop:cm_zerovariance} it follows that
	\eqref{eq:rcm_conv_correlations_3} converges to zero, while this holds for
	\eqref{eq:rcm_conv_correlations_4} because of Proposition~\ref{prop:cm_empirical_conv}
	and the continuous mapping theorem. Finally, since
	\[
		\left(\left(\frac{\Ind{S_n}}{\Prob{S_n}}\right)^2 - 1\right) \le \left(\frac{\Ind{S_n}}{\Prob{S_n}} - 
		1\right)\left(1 + \Prob{S_n}^{-1}\right) \quad \text{and} \quad \CExp{I_{\uedge_n}^{\alpha, \beta}(k, l)}
		{G_n} \le 1,
	\]
	it follows that
	\[
		\eqref{eq:rcm_conv_correlations_2} \le \CExp{\CExp{I_{\uedge_n}^{\alpha, \beta}(k, l)}{G_n}
		\left(\frac{\Ind{S_n}}{\Prob{S_n}} - 1\right)}
		{\bisqn{G_n}}\left(1 + \Prob{S_n}^{-1}\right) \plim 0 \quad \text{as } n \to \infty,
	\]
	by~\eqref{eq:lem:rcm_conv_cond_probability}, Lemma~\ref{lem:rcm_conv_cond_probability} and Proposition 4.4 
	from~\cite{chen2013}. Therefore, using~\eqref{eq:cm_conv_emp_1} and dominated convergence, we get
	\begin{align}
		\lim_{n \to \infty} \Exp{\CExp{I_{\uedge_n}^{\alpha, \beta}(k, l)}{G_n, S_n}^2} 
			= \Prob{\rdegree^\alpha = k}^2\Prob{\rdegree^\beta = l}^2. \label{eq:rcm_conv_correlations_5}
	\end{align}
	Combining~\eqref{eq:rcm_conv_correlations_1} and~\eqref{eq:rcm_conv_correlations_5}, a second moment 
	argument now yields that,
	\[
		\CExp{I_{\uedge_n}^{\alpha, \beta}(k, l)}{G_n, S_n} \plim \Prob{\rdegree^\alpha = k}
		\Prob{\rdegree^\beta = l} \quad \text{as } n \to \infty.
	\]
	The result now follows from Theorem~\ref{thm:convdegdegcorr} by observing that the random 
	variables	$\rdegree^\alpha$ and $\rdegree^\beta$ are independent and not concentrated in a single 
	point. The latter is needed so that in case of average ranking we have $S_{\rdegree^\alpha}\left(
	\rdegree^\alpha\right) \ne 0$, see Theorem~\ref{thm:convdegdegcorr}.
\end{proof}

%% file: configurationmodel_erased.tex
\subsection{Erased Configuration Model}\label{ssec:erased}

When the variances of the degree distributions are infinite, the probability of getting a simple graph 
using RCM converges to zero as the graph size increases. 
To remedy this we use ECM, described in Section 4.2 of~\cite{chen2013}. In ECM stubs are connected 
at random, and then self-loops are removed and multiple edges are merged. We emphasize that for this model 
the actual degree sequence $\degreesqn{G}$ may differ from the 
bi-degree sequence, $\bisqn{G}$, used to do the pairing. 

We will often use results from Proposition 4.5 of~\cite{chen2013}, which we state below for reference. 
\begin{prop}[\cite{chen2013}, Proposition 4.5] \label{prop:repeatedchen}
	Let $G_n = (V_n, E_n)$ be a sequence of ECM graphs with $|V_n| = n$ and $k, l \in \N$. Then, as $n 
	\to \infty$,
	\[ 
		\frac{1}{n} \sum_{v \in V_n} \Ind{D^+ v = k} \plim \Prob{\xi = k} \quad \text{and}
		\quad \frac{1}{n} \sum_{v \in V_n} \Ind{D^- v = l} \plim \Prob{\gamma = l}.
	\]
\end{prop}

We will follow the same second moment argument approach as in the previous section to prove that all 
three rank correlations, $\spearman$, $\spearmanaverage$ and $\kendall$ converge to zero in ECM. 
First we will establish a convergence result for the total number of erased in- and 
outbound stubs.

For $v, w \in V$ and $\alpha \in \{+, -\}$, we denote by $E^{c, \, \alpha}(v)$ and $E^c(v, w)$, 
respectively, the set of erased $\alpha$-stubs from $v$ and erased edges between $v$ and $w$. For 
$e \in V^2$, we write $E^c(e) = E^c(\pi_\ast e, \pi^\ast e)$. 

\begin{lem}\label{lem:erasedstubszero}
	Let $\{G_n\}_{n \in \N}$ be a sequence of ECM graphs with $|V_n| = n$ and $\alpha \in \{+, -\}$. 
	Then
	\[  
		\frac{1}{n} \sum_{v \in V_n} |E_n^{c, \, \alpha}(v)| \plim 0 \quad \text{as } n \to \infty.
	\]
\end{lem}

\begin{proof}
	Let $N \in \N$ and fix a $v \in V_N$, then for all $n \ge N$, $|E_n^{c, \, \alpha}(v)| \le \gamma_n 
	+ 1$ where all $\gamma_n$ are i.i.d. copies of $\gamma$.	Since by Lemma 5.2 from~\cite{chen2013} 
	we have $E_n^{c, \, \alpha}(v) \to 0$ almost surely and furthermore $\Exp{\gamma} < \infty$, 
	dominated convergence implies that
	\[
		\lim_{n \to \infty} \frac{1}{n} \sum_{v \in V_n} \Exp{|E_n^{c, \, \alpha}(v)|} = 0.
	\]
	Applying the Markov inequality then yields, for arbitrary $\varepsilon > 0$,
	\[
		\lim_{n \to \infty }\Prob{\frac{1}{n} \sum_{v \in V_n} |E_n^{c, \, \alpha}(v)| \ge \varepsilon} 
		\le \lim_{n \to \infty} \frac{\sum_{v \in V_n} \Exp{|E_n^{c, \, \alpha}(v)|}}{n \varepsilon} = 0.
	\]
\end{proof}

Since
\[
	|E| = |\widehat{E}| - \sum_{v \in V}|E^{c, \, \alpha}(v)| \quad\mbox{for $\alpha \in \{+, -\}$,}
\] 
the above lemma combined with Proposition~\ref{prop:bisqn_chen} implies that
\begin{equation}
	\frac{|E_n|}{n} \plim \Exp{\gamma} \quad \text{as } n \to \infty.
\label{eq:ecmconvedges}
\end{equation}

We proceed with the next lemma, which is an adjustment of Lemma~\ref{lem:rcmedgecount}, where we now condition on both the bi-degree 
sequence of stubs as well as the eventual degree sequence. We remark that $I_e^{\alpha, \beta}(k, l)$ 
is completely determined by the latter while $\sum_{e \in V^2} |E^c(e)|$ is completely determined by 
the combination of the two sequences. Recall that for $e \in V^2$, $|\widehat{E}(e)|$ denotes the number of edges 
$f \in E$ with $f = e$ before removal of self-loops and merging multiple edges and observe that $|E(e)| = |\widehat{E}(e)| 
- |E^c(e)|$.

\begin{lem}\label{lem:ecmedgecount}
	Let $\{G_n\}_{n \in \N}$ be a sequence of ECM graphs with $|V_n| = n$. 
	Then, for each $k, l \in \N$ and $\alpha, \beta \in \{+, -\}$,
	\begin{enumerate}[\upshape i)]
		\item 
		$\displaystyle
			\CExp{\CExp{I_{\uedge_n}^{\alpha, \beta}(k, l)}{G_n}}{\bisqn{G_n}, 
				\degreesqn{G_n}}
			= \sum_{e \in V_n^2} I_{e}^{\alpha, \beta}(k, l)\frac{D_n^+ \pi_\ast e 
				D_n^- \pi^\ast e}{|E_n|^2} + o_\pr(1),$
		\item
		$\displaystyle
			\CExp{\CExp{I_{\uedge_n}^{\alpha, \beta}(k, l)}{G_n}^2}{\bisqn{G_n}, \degreesqn{G_n}} 
				= \left(\sum_{e \in V_n^2} I_{e}^{\alpha, \beta}
				(k, l)\frac{D_n^+ \pi_\ast e D_n^- \pi^\ast e}{|E_n|^2}\right)^2	+ o_\pr(1).
		$
	\end{enumerate}
\end{lem}

To obtain this result we need the following Lemma.

\begin{lem}\label{lem:ecmbisqn}
	Let $\{G_n\}_{n \in \N}$ be a sequence of ECM graphs with $|V_n| = n$. 
	Then, for each $k, l \in \N$ and $\alpha, \beta \in \{+, -\}$,
	\[
		\sum_{e \in V_n^2} I_e^{\alpha, \beta}(k, l) \frac{\widehat{D}_n^+ \pi_\ast e 
		\widehat{D}_n^- \pi^\ast e}{|\widehat{E}_n|^2} = \sum_{e \in V_n^2} I_e^{\alpha, \beta}
		(k, l) \frac{D_n^+ \pi_\ast e D_n^- \pi^\ast e}{|\widehat{E}_n|^2} + o_\pr(1).
	\]
\end{lem}

\begin{proof}
	Since $\widehat{D}^\alpha_n \pi e = D^\alpha_n \pi e + |E^{c, \, \alpha}_n(\pi e)|$, we have
	\begin{align}
		\sum_{e \in V_n^2} I_e^{\alpha, \beta}(k, l) \frac{\widehat{D}_n^+ \pi_\ast e 
			\widehat{D}_n^- \pi^\ast e}{|\widehat{E}_n|^2} &= \sum_{e \in V_n^2} I_e^{\alpha, \beta}(k, l)
			\frac{D_n^+ \pi_\ast e D_n^- \pi^\ast e}{|\widehat{E}_n|^2} \notag\\
		&\hspace{10pt}+ \sum_{e \in V_n^2} I_e^{\alpha, \beta}(k, l)\frac{\widehat{D}_n^+ \pi_\ast e |E^{c, \, -}
		(\pi^\ast e)|}{|\widehat{E}_n|^2} \label{eq:ecmbisqn_1} \\
		&\hspace{10pt}+ \sum_{e \in V_n^2} I_e^{\alpha, \beta}(k, l)\frac{\widehat{D}_n^- \pi_\ast e |E^{c, \, +}
		(\pi_\ast e)|}{|\widehat{E}_n|^2} \label{eq:ecmbisqn_2}\\
		&\hspace{10pt}+ \sum_{e \in V_n^2} I_e^{\alpha, \beta}(k, l)\frac{|E^{c, \, +}(\pi_\ast e)| |E^{c, \, -}
		(\pi^\ast e)|}{|\widehat{E}_n|^2}. \label{eq:ecmbisqn_3}
	\end{align}
	By Lemma~\ref{lem:erasedstubszero} and Proposition~\ref{prop:bisqn_chen} it follows that~\eqref{eq:ecmbisqn_3} is 
	$o_\pr(1)$. For~\eqref{eq:ecmbisqn_1} we have
	\begin{align*}
		\sum_{e \in V_n^2} I_e^{\alpha, \beta}(k, l)\frac{\widehat{D}_n^+ \pi_\ast e |E^{c, \, -}
			(\pi_\ast e)|}{|\widehat{E}_n|^2} 
		&\le \sum_{v \in V_n} \frac{\widehat{D}^+_n v}
			{|\widehat{E}_n|} \sum_{w \in V_n} \frac{|E_n^{c, -}(w)|}{|\widehat{E}_n|} \\
		&\le \sum_{w \in V_n} \frac{|E_n^{c, -}(w)|}{|\widehat{E}_n|} = o_\pr(1),
	\end{align*}
	where the last line is due to $\sum_{v \in V_n} \widehat{D}^+_n v = |\widehat{E}_n|$. The last equation then follows 
	from Lemma~\ref{lem:erasedstubszero} and Proposition~\ref{prop:bisqn_chen}. This holds similarly for
	\eqref{eq:ecmbisqn_2} and hence the result follows.
\end{proof}

\begin{proof}[Proof of Lemma~\ref{lem:ecmedgecount}.]
	i) By splitting $|E_n(e)|$ we obtain,
	\begin{align}
		\CExp{\CExp{I_{\uedge_n}^{\alpha, \beta}(k, l)}{G_n}}{\bisqn{G_n}, \degreesqn{G_n}}
		&= \CExp{\sum_{e \in V_n^2}
			I_{e}^{\alpha, \beta}(k, l) \frac{|E_n(e)|}{|E_n|}}{\bisqn{G_n}, \degreesqn{G_n}} \notag \\
		&= \frac{|\widehat{E}_n|}{|E_n|} \CExp{\sum_{e \in V_n^2} I_e^{\alpha, \beta}(k, l) \frac{|\widehat{E}_n(e)|}
		{|\widehat{E}_n|}}{\widehat{\mathfrak{D}}(G_n)} \label{eq:ecmempirical1} \\
		&\hspace{10pt}- \frac{1}{|E_n|} \sum_{e \in V_n^2} I_e^{\alpha, \beta}(k, l) 
			\CExp{|E_n^c(e)|}{\bisqn{G_n}, \degreesqn{G_n}} \label{eq:ecmempirical2}
	\end{align}
	For~\eqref{eq:ecmempirical2} we have,
	\begin{align*}
		\frac{1}{|E_n|} \sum_{e \in V_n^2} I_e^{\alpha, \beta}(k, l)\CExp{|E_n^c(e)|}
			{\bisqn{G_n}, \degreesqn{G_n}}
		&\le \frac{1}{|E_n|} \sum_{e \in V_n^2} \CExp{|E_n^c(e)|}
			{\bisqn{G_n}, \degreesqn{G_n}} \\
		&= \frac{1}{|E_n|} \sum_{v \in V_n} |E_n^{c, +}(v)|,
	\end{align*}
	which is $o_\pr(1)$ by Lemma~\ref{lem:erasedstubszero} and~\eqref{eq:ecmconvedges}. Now, since the conditional
	expectation in~\eqref{eq:ecmempirical1} equals~\eqref{eq:cm_firstmoment}, it follows from Lemma
	\ref{lem:rcmedgecount} i), Lemma~\ref{lem:ecmbisqn} and~\eqref{eq:ecmconvedges} that
	\[
		\eqref{eq:ecmempirical1} = \sum_{e \in V_n^2} I_{e}^{\alpha, \beta}(k, l)\frac{D_n^+ \pi_\ast e D_n^- 
		\pi^\ast e}{|E_n|^2} + o_\pr(1).
	\]
	
	ii)	Splitting both terms $|E_n(e)|$ and $|E_n(f)|$ for $e, f \in V_n^2$ yields,
	\begin{align}
		&\CExp{\CExp{I_{\uedge_n}^{\alpha, \beta}(k, l)}{G_n}^2}{\bisqn{G_n}, \degreesqn{G_n}} \notag\\
		&=\CExp{\sum_{e, f \in V_n^2}I_{e}^{\alpha, \beta}(k, l) I_{f}^{\alpha, \beta}(k, l) 
			\frac{|E_n(e)| \, |E_n(f)|}{|E_n|^2}}{\bisqn{G_n}, 
			\degreesqn{G_n}} \notag\\
		&= \frac{|\widehat{E}_n|^2}{|E_n|^2} \CExp{\sum_{e, f \in V_n^2} I_e^{\alpha, \beta}(k, l)
			I_f^{\alpha, \beta}(k, l)\frac{|\widehat{E}(e)||\widehat{E}(f)|}{|\widehat{E}_n|}}{\bisqn{G_n}} 
			\label{eq:ecmempirical3}\\
		&\hspace{10pt}+ \sum_{e, f \in V_n^2} I_e^{\alpha, \beta}(k, l)I_f^{\alpha, 
			\beta}(k, l)\CExp{\frac{|E_n^c(e)| |E_n^c(f)|}
			{|E_n|^2}}{\bisqn{G_n}, \degreesqn{G_n}} 
			\label{eq:ecmempirical4} \\
		&\hspace{10pt}- \sum_{e, f \in V_n^2} I_e^{\alpha, \beta}(k, l)
			I_f^{\alpha, \beta}(k, l)\CExp{\frac{|E_n^c(e)||\widehat{E}_n(f)|}{|E_n|^2}}{\bisqn{G_n}, 
			\degreesqn{G_n}} \label{eq:ecmempirical5}\\
		&\hspace{10pt}- \sum_{e, f \in V_n^2} I_e^{\alpha, \beta}(k, l)
			I_f^{\alpha, \beta}(k, l)\CExp{\frac{|E_n^c(f)||\widehat{E}_n(e)|}{|E_n|^2}}{\bisqn{G_n}, 
			\degreesqn{G_n}} \label{eq:ecmempirical6}
	\end{align}
	Recognizing the conditional expectation in~\eqref{eq:ecmempirical3} as~\eqref{eq:cm_secondmoment}, 
	then using first Lemma~\ref{lem:rcmedgecount} ii) and then Lemma~\ref{lem:ecmbisqn} and
	\eqref{eq:ecmconvedges}, it follows that~\eqref{eq:ecmempirical3} equals
	\[
			\left(\sum_{e \in V_n^2} I_{e}^{\alpha, \beta}(k, l)\frac{D_n^+ \pi_\ast e 
			D_n^- \pi^\ast e}{|E_n|^2}\right)^2 + o_\pr(1).
	\]
	It remains to show that~\eqref{eq:ecmempirical4}-\eqref{eq:ecmempirical6} are $o_\pr(1)$. 
	For~\eqref{eq:ecmempirical4} we have 
	\begin{align*}
		\sum_{e, f \in V_n^2} I_e^{\alpha, \beta}(k, l)I_f^{\alpha, \beta}(k, l)
			\CExp{\frac{|E_n^c(e)| |E_n^c(f)|}{|E_n|^2}}{\bisqn{G_n}, \degreesqn{G_n}}
		&\le \left(\frac{1}{|E_n|}\sum_{v \in V_n}|E_n^{c, \, +}(v)|\right)^2 = o_\pr(1)
	\end{align*}
	by Lemma~\ref{lem:erasedstubszero} and~\eqref{eq:ecmconvedges}. Since~\eqref{eq:ecmempirical5} 
	and~\eqref{eq:ecmempirical6} are symmetric we will only consider the	latter:
	\begin{align*}
		&\sum_{e, f \in V_n^2} I_e^{\alpha, \beta}(k, l)I_f^{\alpha, \beta}(k, l)
			\CExp{\frac{|E_n^c(f)||\widehat{E}_n(e)|}{|E_n|^2}}
			{\bisqn{G_n}, \degreesqn{G_n}} \\
		&\le \left(\sum_{f \in V_n^2}\frac{|E_n^c(f)|}{|E_n|}\right) 
			\frac{1}{|E_n|}\sum_{e \in V_n^2}\CExp{|\widehat{E}_n(e)|}{\bisqn{G_n}} \\
		&= \left(\sum_{v \in V_n} \frac{|E_n^+(v)|}{|E_n|}\right)
			\frac{|\widehat{E}_n|}{|E_n|} = o_\pr(1).
	\end{align*}
	Here, for the last line, we used $\sum_{e \in V_n^2}\CExp{|\widehat{E}_n(e)|}{\bisqn{G_n}}
	= |\widehat{E}_n|$, and then Lemma~\ref{lem:erasedstubszero} and~\eqref{eq:ecmconvedges}. 
\end{proof}

A straightforward adaptation of the proof of Proposition~\ref{prop:cm_empirical_conv}, using
Lemma~\ref{lem:ecmedgecount} instead of Lemma~\ref{lem:rcmedgecount}, yields the 
following result.

\begin{prop}\label{prop:ecmempiricalconv}
	Let $\{G_n\}_{n \in \N}$ be a sequence of ECM graphs with $|V_n| = n$ and $\alpha, \beta \in \{+, -\}$. 
	Then there exist integer valued random variables $\rdegree^\alpha$ and $\rdegree^\beta$ such that for 
	each $k, l \in \N$, as $n \to \infty$,
	\[
		\CExp{\CExp{I_{\uedge_n}^{\alpha, \beta}(k, l)}{G_n}}{\widehat{\mathfrak{D}}(G_n), \mathfrak{D}(G_n)}
		\plim \Prob{\rdegree^\alpha = k}\Prob{\rdegree^\beta = l},
	\]
	where the distributions of $\rdegree^\alpha$ and $\rdegree^\beta$ are given in Table
	\ref{tab:limiting_degrees}.
\end{prop}

We can now again use a second moment argument to get the convergence result for the three rank 
correlations in the Erased Configuration Model. We omit the proof since the computation of the variance
follows the exact same steps as those in Proposition~\ref{prop:cm_conv_empirical_propabilities}, where now, 
instead of only conditioning on $\bisqn{G_n}$, we also condition on $\degreesqn{G_n}$ and use Lemma
\ref{lem:ecmedgecount}. 

\begin{thm}
	Let $\{G_n\}_{n \in \N}$ be a sequence of ECM graphs with $|V_n| = n$ and $\alpha, \beta \in 
	\{+, -\}$. Then, as $n \to \infty$,
	\[
		\spearman_\alpha^\beta(G_n) \plim 0, \quad \spearmanaverage_\alpha^\beta(G_n) \plim 0 \quad 
		\text{and} \quad \kendall_\alpha^\beta(G_n) \plim 0.
	\]
\end{thm}

This theorem shows that even when the variance of the degree sequences is infinite, one can 
construct a random graph for which the degree-degree dependencies, measured by rank correlations, 
converge to zero in the infinite graph size limit. Therefore this model can be used as a null 
model for such dependencies.

%% file: appendix.tex
\begin{appendices}

\section{Continuization}\label{sec:cont}

In this appendix we will establish several relations between the distribution functions of 
integer valued random variables and their continuizations, using the functions $\mathcal{F}$ and 
$\mathcal{H}$ defined in \eqref{eq:scriptF} and \eqref{eq:scriptH}, respectively.

Let $\cont{X} = X + U$ be as in Definition~\ref{def:cont}, take $k \in \Z$ and define $I_k = [k, 
k + 1)$. Then for $x \in I_k$,
\begin{equation}
	F_{\cont{X}}(x) = (x - k)F_X(k) + (k + 1 - x)F_X(k - 1).
\label{eq:contdistribution}
\end{equation}
As a consequence, it follows that for $x \in I_k$,
\begin{equation}
	dF_{\cont{X}}(x) = \left(F_X(k) - F_X(k - 1)\right)dx = \Prob{X = k}dx.
\label{eq:contdensity}
\end{equation}

These identities capture the essential relations between $X$ and its continuization $\cont{X}$. As a
first result we have the following.

\begin{lem}\label{lem:contmoments}
	Let $X$ be an integer valued random variable and $m \in \N$. Then, 
	\[
		\Exp{F_{\cont{X}}(\cont{X})^m} = \frac{1}{m + 1} \sum_{i = 0}^m \Exp{F_X(X)^i 
			F_X(X - 1)^{m - i}}.
	\]
\end{lem}

\begin{proof}
	Using~\eqref{eq:contdistribution} we obtain,
	\begin{align*}
		\int_{I_k} F_{\cont{X}}(x)^m dx 
		&= \int_{I_k} \left((x - k)F_X(k) + (k + 1 - x)F_X(k - 1)\right)^m dx \\
		&= \sum_{i = 0}^m \binom{m}{i} F_X(k)^i F_X(k - 1)^{m - i} \int_0^1 (y)^i(1 - y)^{m - i} dy \\
		&= \sum_{i = 0}^m \frac{m!}{i!(m - i)!} F_X(k)^i F_X(k - 1)^{m - i} 
			\frac{\Gamma(i + 1)\Gamma(m - i + 1)}{\Gamma(m + 2)} \\
		&= \frac{1}{m + 1} \sum_{i = 0}^m F_X(k)^i F_X(k - 1)^{m - i},
	\end{align*}
	Combining this with~\eqref{eq:contdensity}, we get
	\begin{align*}
		\Exp{F_{\cont{X}}(\cont{X})^m} 
		&= \sum_{k \in \Z} \int_{I_k} F_{\cont{X}}(x)^m dF_{\cont{X}}(x) \\
		&= \sum_{k \in \Z} \int_{I_k} F_{\cont{X}}(x)^m \Prob{X = k} dx \\
		&= \frac{1}{m + 1} \sum_{i = 0}^m \Exp{F_X(X)^i F_X(X - 1)^{m - i}}.
	\end{align*}
\end{proof} 
As a direct consequence of Lemma~\ref{lem:contmoments} we get
\begin{equation}\label{eq:EF}
	\frac{1}{2} = \Exp{F_{\cont{X}}(\cont{X})} = \frac{1}{2}\Exp{\mathcal{F}_X(X)},
\end{equation}
relating $F_{\cont{X}}$ to $\mathcal{F}_X$. Similar to~\eqref{eq:contdistribution}, if $Z$ is
a random element independent of $X$, we get for $x \in I_k$,

\begin{align}
	F_{\cont{X}|Z}(x) = (x - k)F_{X|Z}(k) + (k + 1 - x)F_{X|Z}(k - 1). \label{eq:condcontdistr}
\end{align}

Applying~\eqref{eq:condcontdistr} in a similar way as~\eqref{eq:contdistribution} we arrive at an 
extension of Lemma~\ref{lem:contmoments}. The proof is elementary, hence omitted.

\begin{prop}\label{prop:condcontprop}
	Let $X$ be an integer valued random variable and $Z$ a random element independent of the 
	continuous part of $\cont{X}$. Then
	\begin{enumerate}[\upshape i)]
		\item 
			$\displaystyle{\CExp{F_{\cont{X}}(\cont{X})}{Z} = \frac{1}{2}\CExp{
			\mathcal{F}_X(X)}{Z}}$, a.s.;
		\item
			$\displaystyle{F_{\cont{X}|Z}\left(\cont{X}\right) = \frac{1}{2}
			\mathcal{F}_{X|Z}(X)}$, a.s.
	\end{enumerate}
\end{prop}

The following results are extensions of the previous ones to the case of two integer valued random 
variables $X$ and $Y$. We will state these without proofs, since these are either straightforward
extensions of those for the case of a single random variable or follow from elementary calculations 
and the previous results. 

\begin{lem}\label{lem:contjointmoments}
	Let $X, Y$ be integer valued random variables. Then,
	\begin{enumerate}[\upshape i)]
	\item 
		$\displaystyle{\Exp{F_{\cont{X}}(\cont{X})F_{\cont{Y}}(\cont{Y})} 
			= \frac{1}{4}\Exp{\mathcal{F}_X(X)\mathcal{F}_Y(Y)}}$,
	\item
		$\displaystyle{\Exp{H_{\cont{X}, \cont{Y}}(\cont{X}, \cont{Y})} = 
		\frac{1}{4}\Exp{\mathcal{H}_{X, Y}(X, Y)}}$.
	\end{enumerate}
\end{lem}

\begin{prop}\label{prop:condcontjointprop}
	Let $X, Y$ be integer valued random variables and let $Z$ be a random	variable independent of 
	the uniform parts of $\widetilde{X}$ and $\widetilde{Y}$. Then
	\begin{enumerate}[\upshape i)]
		\item 
			$\displaystyle{\CExp{\cont{F}_{\cont{X}}(\cont{X})\cont{F}_{\cont{Y}}(\cont{Y})}{Z} 
				= \frac{1}{4}\CExp{\mathcal{F}_X(X)\mathcal{F}_Y(Y)}{Z}}$ a.s.;
		\item
			$\displaystyle{H_{\cont{X}, \cont{Y}|Z}(\cont{X}, \cont{Y}) = 
			\frac{1}{4}\mathcal{H}_{X, Y|Z}(X, Y)}$ a.s.
	\end{enumerate}
\end{prop}

\end{appendices}